\documentclass[stropt,envcountsect,envcountsame]{svmult}

\usepackage{amssymb,latexsym}

\usepackage{amsmath}
\usepackage{graphicx}
\usepackage{lipsum}
\usepackage{bbm}
\smartqed
\renewcommand{\setitemindent}{-.5in}

\begin{document}

\author{Gregory J. Morrow}

\institute{Gregory J. Morrow \at
University of Colorado, Colorado Springs CO 80918  \email gmorrow@uccs.edu}

\title{Laws relating runs, long runs, and steps in gambler's ruin, with persistence in two strata} 
\titlerunning{Persistent Gambler's Ruin}

\maketitle

\abstract{
Define a certain gambler's ruin process $\mathbf{X}_{j}, \mbox{ \ }j\ge 0,$ such that the increments $\varepsilon_{j}:=\mathbf{X}_{j}-\mathbf{X}_{j-1}$ take values $\pm1$ and satisfy $P(\varepsilon_{j+1}=1|\varepsilon_{j}=1, |\mathbf{X}_{j}|=k)=P(\varepsilon_{j+1}=-1|\varepsilon_{j}=-1,|\mathbf{X}_{j}|=k)=a_k$, all $j\ge 1$,  where $a_k=a$ if $ 0\le k\le f-1$, and $a_k=b$ if $f\le k<N$. Here  $0<a, b <1$ denote persistence parameters and $ f ,N\in \mathbb{N} $ with $f<N$. The process starts at $\mathbf{X}_0=m\in (-N,N)$ and terminates when $|\mathbf{X}_j|=N$. Denote by ${\cal R}'_N$, ${\cal U}'_N$, and ${\cal L}'_N$, respectively, the numbers of runs,  long runs, and steps in the meander portion of the gambler's ruin process.  Define $X_N:=\left ({\cal L}'_N-\frac{1-a-b}{(1-a)(1-b)}{\cal R}'_N-\frac{1}{(1-a)(1-b)}{\cal U}'_N\right )/N$ and let $f\sim\eta N$ for some $0<\eta <1$. We show $\lim_{N\to\infty} E\{e^{itX_N}\}=\hat{\varphi}(t)$ exists in an explicit form. We obtain a companion theorem for the last visit portion of the gambler's ruin.
}
\section{Introduction}\label{S:intro}
Define a gambler's ruin process $\{\mathbf{X}_{j}, \mbox{ \ }j\ge 0\},$ with values in $\mathbb{Z}\cap[-N,N]$, such that the increments $\varepsilon_{j}:=\mathbf{X}_{j}-\mathbf{X}_{j-1}$ take values $\pm1$ and satisfy $P(\varepsilon_{j+1}=1|\varepsilon_{j}=1, |\mathbf{X}_{j}|=k)=P(\varepsilon_{j+1}=-1|\varepsilon_{j}=-1,|\mathbf{X}_{j}|=k)=a_k$, all $j\ge 1$,  where $a_k=a$ if $ 0\le k\le f-1$, and $a_k=b$ if $f\le k<N$. Here  $0<a, b <1$ denote persistence parameters and $ f ,N\in \mathbb{N} $ with $f<N$.
The process starts at some fixed level $m\in (-N,N)$ and terminates at an epoch $j$ when $|\mathbf{X}_j|=N$.  For initial probabilities, take $\pi_{+}=P(\varepsilon_{j}=1)=\pi_{-}=P(\varepsilon_{j}=-1)=\frac{1}{2}$.
We call the two ranges of values $|k|\le f-1$ and  $f\le |k|<N$ as strata for the two persistence parameter values $a$ and $b$, respectively. In gambling, $\mathbf{X}_j$ denotes a fortune after $j$ games on which the gambler makes unit bets.  If $a,b >\frac{1}{2}$, then any run of fortune tends to keep going in the same direction. Thus for example a win [loss] resulting in  fortune  $k$ for some  $|k|\le f-1$ is followed by another win [loss] with probability $a$, whereas a change in fortune occurs with probability $1-a$.  
Henceforth we shall simply refer to $\{\mathbf{X}_j=\mathbf{X}_j^{N}\}$ as the gambler's ruin process, with or without mention of the parameters $a,$ $b$, $f$, and $N$. Note that $\{\mathbf{X}_j\}$  is the classical fair gambler's ruin process in case $a=b=\frac{1}{2}$, with symmetric boundaries $N$ and $-N$. For the \emph{homogeneous} case $a=b$, the increments $\{\varepsilon_{j}, j\ge 0\},$ form a strictly stationary process with zero means, where the correlation between $\varepsilon_{j}$ and $\varepsilon_{j+1}$ is  $2a-1$. If $a=b$ and also $N=\infty$ then $\{\mathbf{X}_j^{\infty}\}$ becomes a symmetric persistent (or correlated) random walk on $\mathbb{Z}$  
that is recurrent by \cite{Sp1964}, Thm. 8.1. 
\par
Physical models of persistence often consider the velocity of a particle either staying the same or being changed according to a random collision process \cite{BaToTo2007, PoKi2016, SzTo1984}; in our model the velocity only takes values $\pm 1$. 
Our introduction of strata corresponds to a change in medium over which the persistence parameter, or likelihood of the velocity staying the same, would deterministically change. In \cite{SzTo1984}, the authors obtain a Wiener limit for the normalized sum of velocities under a random environment, that includes our deterministic model. Our aim is different since we want results for discrete statistics that have no analogue in the Wiener process. In this context our stratified model seems to be new.
\par
We define a nearest neighbor path of length $n$ in $\mathbb{Z}$ to be a sequence $\Gamma=\Gamma_0,\Gamma_1,\dots,\Gamma_n,$ where $\Gamma_j\in \mathbb{Z}$ and $\delta_j:=\Gamma_j-\Gamma_{j-1}$ satisfies $|\delta_j|=1$ for all $j=1,\dots,n$. We also call $n$ the number of steps of $\Gamma$. We connect successive lattice points $(j-1,\Gamma_{j-1})$ and $(j,\Gamma_{j})$ in the plane by straight line segments, and term this connected union of straight line segments the \emph{lattice path}. See Figures \ref{F:lastvisit}--\ref{F:futuremaxima1}.
We define the number of runs along $\Gamma$ as the number of inclines, either straight line ascents or descents, of maximal extent along the lattice path; the length of a run is the number of steps in such a maximal ascent or descent.  A  \emph{long run} is itself a run that  consists of at least two steps; in gambling terminology a long run means that the run of fortune does not immediately change direction.  A \emph{short run} is on the other hand  a run of length exactly one, so every run is either a long run or a short run.  In Figure \ref{F:futuremaxima1}, the lattice path shown has 15 runs, with 7 short runs and 8 long runs. An \emph{excursion} is a nearest neighbor path that starts and ends at $m=0$, $\Gamma_0=\Gamma_n=0$, but for which $\Gamma_j\neq 0$ for $1\le j\le n-1$. A positive excursion is an excursion whose graph lies above the $x$--axis save for its endpoints. For a positive excursion path, the number of runs is just twice the number of peaks, where a peak at lattice point $(j,\Gamma_j)$ corresponds to $\delta_{j}=1$ and $\delta_{j+1}=-1$. 
\par
The \emph{last visit} is defined as 
\begin{equation}\label{E:LN}
{\cal L}_N:= \max\{j\ge 0: j=0, \mbox{ or } \mathbf{X}_ j= 0 \mbox{ for some  } j\ge 1\}.
\end{equation}
The \emph{meander} is the portion of the process that extends from the epoch of the last visit ${\cal L}_N$ until the gambler's ruin process terminates.  So the meander process never returns to the level $m=0$. See Figure \ref{F:lastvisit}.
\begin{figure}
\begin{center}
\setlength{\unitlength}{0.0075in}%
\begin{picture}(300,140)(100,710)
\thinlines
\put(-11,775){0}
\put(-11,795){1}
\put(-11,815){2}
\put(-11,835){3}
\put(-11,855){4}
\put(-16,755){-1}
\put(-16,735){-2}
\put(-16,715){-3}
\put(-16,695){-4}
\put(00,780){\line(1,-1){20}}
\put(20,760){\line(1,-1){20}}
\put(40,740){\line(1,1){20}}
\put(60,760){\line(1,1){20}}
\put(80,780){\line(1,-1){20}}
\put(100,760){\line(1,1){20}}
\put(120,780){\line(1,1){20}}
\put(140,800){\line(1,1){20}}
\put(160,820){\line(1,1){20}}
\put(180,840){\line(1,-1){20}}
\put(200,820){\line(1,-1){20}}
\put(220,800){\line(1,-1){20}}
\put(240,780){\line(1,-1){20}}
\put(260,760){\line(1,-1){20}}
\put(280,740){\line(1,1){20}}
\put(300,760){\line(1,-1){20}}
\put(320,740){\line(1,1){20}}
\put(340,760){\line(1,1){20}}
\put(360,780){\line(1,1){20}}
\put(380,800){\line(1,1){20}}
\put(400,820){\line(1,1){20}}
\put(420,840){\line(1,-1){20}}
\put(440,820){\line(1,1){20}}
\put(460,840){\line(1,1){20}}
\put(355,760){${\cal L}_4$}
\put(405,795){Meander}
\put(365,785){\dots\dots\dots\dots\dots\dots}
\linethickness{0.050mm}
\put(00,860){\line(1,0){500}}
\put(00,700){\line(1,0){500}}
\thicklines
\put(00,780){\line(1,0){500}}
\put(00,700){\line(0,1){160}}
\end{picture}
\end{center}
\caption{Last Visit and Meander; $N=4$.} 
\label{F:lastvisit}
\end{figure}
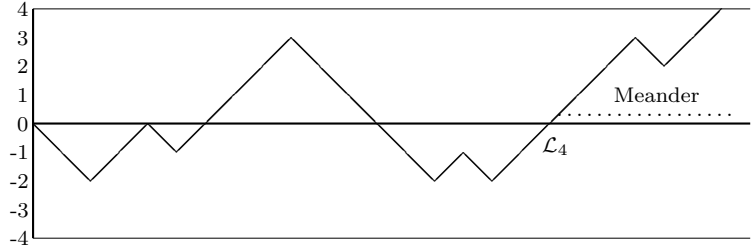 
It is shown by \cite{Mor2015} that, for $a=b=\frac{1}{2}$, if  ${\cal R}_N$ denotes the total number of runs over all excursions of the absolute value process $\{|\mathbf{X}_j|\}$ until the last visit, then, with the 
order $N$ scaling, it holds that  $({\cal L}_N-2{\cal R}_N)/N$ converges in law. Also, if ${\cal R}'_N$ and ${\cal L}'_N $ denote respectively the number of runs and steps over the meander portion of the process, then  $({\cal L}'_N-2{\cal R}'_N )/N$ converges in law to a density $\varphi(x) = ( \pi /4) \mbox{sech} ^2 (\pi x/2)$, $-\infty<x<\infty$, with characteristic function $\int_{-\infty}^{\infty}\varphi(x)e^{ixt} dx=t/\sinh(t)$. We first generalize this result  for the meander case. 
 Let ${\cal R}'_N$, ${\cal V}'_N$, ${\cal L}'_N$,  denote respectively the numbers of runs, short runs, and steps, in the meander portion of the gambler's ruin; for the lattice path of Figure \ref{F:lastvisit} we have ${\cal R}'_4=3,$ ${\cal V}'_4=1$, ${\cal L}'_4=6$. Define the following scaled random variable over the meander:
\begin{equation}\label{E:Xab} 
\begin{array}{c}
X_N:=\frac{1}{N}\left ({\cal L}'_N-\frac{2-a-b}{(1-a)(1-b)}{\cal R}'_N+\frac{1}{(1-a)(1-b)}{\cal V}'_N\right ).
\end{array}
\end{equation}
\begin{theorem}\label{T:T2}
 Let $f=\eta N$ for some fixed $0<\eta <1$. Denote $\kappa_1:=\frac{\eta \sigma_1}{1-b}$ and $\kappa_2:=\frac{(1-\eta)\sigma_2}{1-a}$, with $\sigma_1=\sqrt{a+b^2-2ab}$ and $\sigma_2=\sqrt{b+a^2-2ab}$. Let $X_N$ be defined by \eqref{E:Xab}. Then, $\lim\limits_{N\to\infty} E\{e^{itX_N}\}=\hat{\varphi}(t), \mbox{ where }$\\ $ (b\kappa_1\sigma_2+a\kappa_2\sigma_1)t/\hat{\varphi}(t):= $
\begin{equation}\label{E:T2}
a\sigma_1\cosh(\kappa_1 t)\sinh(\kappa_2 t)+b\sigma_2\sinh(\kappa_1 t)\cosh(\kappa_2 t)+i (b-a)^2\sinh(\kappa_1 t)\sinh(\kappa_2 t).
\end{equation}
\end{theorem}
In Theorem \ref{T:T2}, we obtain that $\varphi(x):=\frac{1}{2\pi}\int_{-\infty}^{\infty}e^{-itx}\hat{\varphi}(t)dt$ is real since the complex conjugate of $\varphi(x)$ is equal to itself; observe this by 
making a change of variables $t\to-t$ after conjugation of the integral. 
\par
We have a bivariate result for the homogeneous case as follows.  Define
\begin{equation}\label{E:Delta1Delta2}
\begin{array}{c}
Y_{1,N}:=\frac{1}{N}\left({\cal R}'_N-\frac{1}{(1-a)}{\cal V}'_N\right); \ \  
Y_{2,N}:=\frac{1}{N}\left({\cal L}'_N-\frac{1}{(1-a)}{\cal R}'_N\right)-Y_{1,N}.
\end{array}
\end{equation}
\begin{corollary}\label{C:homogeneouslimit}
Suppose $a=b$. Then the limiting joint characteristic function of the random variables $Y_{1,N}$ and $Y_{2,N}$ is: 
\begin{equation}\label{E:Jointchfcn}
\lim_{N\to\infty} E\{e^{isY_{1,N}+itY_{2,N}}\}=
\frac{\sqrt{(1-a)s^2+at^2}}{\sinh (\sqrt{(1-a)s^2+at^2})}. 
\end{equation}
\end{corollary}
\begin{remark}\label{R:Xzeta}
Let $a=b$, and define $X_{\zeta , N}:=\frac{1}{N}\left ({\cal L}'_N-\frac{1+\zeta}{(1-a)}{\cal R}'_N+\frac{\zeta}{(1-a)^2}{\cal V}'_N \right )$. Then by setting $s=(1-a-\zeta)t/(1-a)$ in \eqref{E:Jointchfcn} we obtain, for all $\zeta\in \mathbb{R}$,
$$ \lim_{N\to\infty} E\{e^{itX_{\zeta, N}}\}=\frac{A_{\zeta}t}{\sinh (A_\zeta t)}; \ \ A_\zeta:=\sqrt{[(2\zeta -1)a +(1-\zeta)^2]/(1-a)}.$$
\end{remark}
 \begin{example}\label{R:asymmetric-vhf}   As a special case of Theorem \ref{T:T2}, consider $b=1-a$ and $\eta=a$. Then $\kappa_i=\sigma_j=\sigma:=\sqrt{1-3a+3a^2}$, for all $i,j=1,2$. In this case we have 
\begin{equation}\label{E:phihat}
\hat{\varphi}(t)=\sigma^2 t/\left \{\sinh(\sigma t)[\sigma \cosh(\sigma t)+i (1-2a)^2\sinh(\sigma t)]\right\}.
\end{equation}
The complex factor of the denominator of $\hat{\varphi}(t)$ in \eqref{E:phihat} is equal to zero if and only if $e^{2\sigma t}=\frac{-\sigma+(1-2a)^{2}i}{\sigma+(1-2a)^{2}i}$. 
The smallest root is $t=\frac{i}{2\sigma}(\pi-\arctan \frac{2\sigma(1-2a)^2}{\sigma^2-(1-2a)^4} )$, with $\sigma^2-(1-2a)^4  =a(1-a)(5-16a+16a^2)>0$. Thus we can analytically continue $\hat{\varphi}(t)$ to a suitably chosen ball of positive radius $\epsilon_0$ about the origin such that $\sup_{|\xi|\le \epsilon_0}\|\hat{\varphi}(\cdot+i\xi)\|_2 <\infty$. It follows by \cite[Thm. IX.13]{ReSi1972}  that the inverse Fourier transform $\varphi(x)$ of $\hat{\varphi}(t)$ has exponential decay, meaning $e^{\epsilon |x|}\varphi(x)$ is square integrable for any $\epsilon<\epsilon_0$. 
However the probability density, $\varphi(x)$, is not symmetric in $x$ under \eqref{E:phihat} with $a\neq \frac{1}{2}$; see Figure \ref{F:density} at the end of the paper; see also \cite{Mor2018} for computational details.
\end{example}
\par
We now introduce the definitions of the excursion statistics to further describe our results. For the definitions in this paragraph we assume $\mathbf{X}_{0}=0$ and $N=\infty$.
Define the index $j$, or step, of first return of $\{\mathbf{X}_j\}$ to the origin by $\mathbf{L}:= \inf\{j\ge 1: \mathbf{X}_j=0\}$. Define the excursion sequence from the origin by $\mathbf{\Gamma}:=\{\mathbf{X}_j, j=0,\dots, \mathbf{L}\}$; again $\mathbf{L}$ is the number of steps of $\mathbf{\Gamma}$. Define the \emph{height} $\mathbf{H}$ of the excursion $\mathbf{\Gamma}$ as the maximum absolute value of the path over this excursion:   
\begin{equation}\label{E:height}
\mathbf{H}:=\max \{|\mathbf{X}_j|: j=1,\dots,\mathbf{L} \}. 
\end{equation} 
Also define $\mathbf{R}$ as the number of runs along $\mathbf{\Gamma}$, and further define $\mathbf{V}$ as the number of short runs along $\mathbf{\Gamma}$. Thus officially $\mathbf{U}:=\mathbf{R}-\mathbf{V}$ is the number of long runs along $\mathbf{\Gamma}$. In Figure \ref{F:lastvisit} there are 4 excursions until the last visit to the origin, with respective heights: $2,\ 1,\ 3, \ 2$. The numbers of runs in  the excursions of the absolute value process $\{|\mathbf{X}_j|\}$ until the last visit of Figure \ref{F:lastvisit}, wherein negative excursions are reflected into positive excursions, are: $2,\ 2,\ 2, \ 4$. The corresponding numbers of short runs in this last visit portion of the absolute value process are: $0,\ 2,\ 0, \ 2$. 
\par
The first motivation of the present paper is to show how the method of \cite{Mor2015} extends to the three statistics, runs, short runs, and steps, in the homogeneous setting ($a=b$). As a particular result we find the following Corollary \ref{C:RUL3waysymm}, which connects the present work with a certain combinatorial domain in the study of Dyck paths. Note that the generating function method which drives the present study depends heavily on a \emph{return to the level} 1 type recurrence approach that has been applied extensively in the field of lattice path combinatorics;   \cite{BaFl2002, BaNi2010, Bou-Me2008, Deu1999, Fl1980, FlSe2009, Kr2015}. 
Let $P_a$ denote the probability for the homogeneous model with persistence parameter $a$. We obtain the following symmetry for the joint distribution of the excursion statistics.
\begin{corollary}\label{C:RUL3waysymm}
Let $a=b$ and assume $\mathbf{X}_0=0$ and $N=\infty$. Then for all $n\ge 2$ there holds:
\begin{equation}\label{E:RUL3waysymm}
(1-a)P_a(\mathbf{L}=2n, \mathbf{R}=2k, \mathbf{U}=\ell)=aP_{1-a}(\mathbf{L}=2n, \mathbf{L}-\mathbf{R}=2k, \mathbf{U}=\ell).
\end{equation}
In particular if $a=\frac{1}{2}$, then $E\{e^{i r\mathbf{R}}e^{i s\mathbf{U}}e^{i t\mathbf{L}}\}-E\{e^{i r(\mathbf{L}-\mathbf{R})}e^{i s\mathbf{U}}e^{i t\mathbf{L}}\}=\frac{1}{2}e^{2it}(e^{2ir}-1).$
\end{corollary}
The Corollary \ref{C:RUL3waysymm} extends the known result for the simple symmetric random walk that $P(\mathbf{L}
=2n,\mathbf{R}=2k)=P(\mathbf{L}=2n,\mathbf{L}-\mathbf{R}=2k)$, $n\ge 2$. Our proof depends on algebraic manipulation of the generating function; see Section \ref{S:Appl}.
\par
The second motivation is to extend the persistence model to the case of two distinct strata $a\neq b$.  
This \emph{full model}, together with its solution, has interesting features, which include: 
\begin{enumerate}
\item  its intrinsic value as physical model; cf. \cite{BaToTo2007, SzTo1984},  
\item completely explicit formulae throughout for key polynomials, identities, and generating functions; 
\item new limiting distributions for a scaling of order $N$ in both the meander and the last visit portions of the gambler's ruin.
\end{enumerate} 
We finally state a companion result to Theorem \ref{T:T2}, again for the full model, that gives a scaling limit of order $N$ over the last visit portion of the gambler's ruin. Let ${\cal R}_N$ and $ {\cal V}_N$ denote the total number of runs and short runs of the absolute value process $\{|\mathbf{X}_j|\}$ until the epoch of the last visit, ${\cal L}_N$, defined by \eqref{E:LN}. Define ${\cal M}_N$ as the number of consecutive excursions of height at most $N-1$ of the absolute value process $\{|\mathbf{X}_j|\}$ until ${\cal L}_N$.
In Figure~\ref{F:lastvisit}, we have ${\cal M}_4=4$, ${\cal R}_4=10,$ $ {\cal V}_4=4$, ${\cal L}_4=18.$
Define:
\begin{equation}\label{E:lastvisitXab} 
\begin{array}{c}
{\cal X}_N:=\frac{1}{N}\left ({\cal L}_N-\frac{2-a-b}{(1-a)(1-b)}{\cal R}_N+\frac{1}{(1-a)(1-b)}{\cal V}_N -\frac{a(b-a)}{(1-a)(1-b)}{\cal M}_N \right ).
\end{array}
\end{equation}
\begin{theorem}\label{T:T3}
Let $f \sim\eta N$, as $N\to \infty$, for some fixed $0<\eta <1$. Let ${\cal X}_N$ be defined by \eqref{E:lastvisitXab}. Let also $\kappa_j$, $j=1,2$, and $\sigma_j$, $j=1,2$, be as defined in Theorem \ref{T:T2}. Let $\hat{\varphi}$ be defined by \eqref{E:T2}. Then,  
$\lim\limits_{N\to\infty} E\{e^{it{\cal X}_N}\}=\hat{\psi}(t)/\hat{\varphi}(t),$ where \\ 
$\begin{array}{c}
(ab\sigma_1\sigma_2)/\hat{\psi}(t):= a b \sigma_1\sigma_2\cosh(\kappa_1 t)\cosh(\kappa_2 t)+a^2\sigma_1^2\sinh(\kappa_1 t)\sinh(\kappa_2 t) \\  +ia\sigma_1(b-a)^2\cosh(\kappa_1 t)\sinh(\kappa_2 t).
\end{array}$
\end{theorem}
Theorem \ref{T:T2}, Corollary  \ref{C:homogeneouslimit}, and  Theorem  \ref{T:T3}, are proved in Section \ref{S:T2}, naturally following  Section \ref{S:proofs} on building blocks for the proofs. 
\section{Elements of the proof.}\label{S:elements}
Recall the definitions of the excursion statistics in \eqref{E:height} and following.  
We define the conditional joint probability generating function of the excursion statistics for runs, short runs, and steps given the height is at most $N$ by
\begin{equation}\label{E:KN}
K_N(a,b):=E\{r^{\mathbf{R}}y^{\mathbf{V}}z^{\mathbf{L}}|\mathbf{X}_0=0,\ \mathbf{H}\le N\}.
\end{equation} 
To calculate \eqref{E:KN}, our proofs feature bivariate Fibonacci polynomials $\{q_n(x,\beta)\}$ and $\{w_n(x,\beta)\}$, defined as follows.
\begin{definition}\label{D:wdef}
Define sequences $q_n(x,\beta)$ and $w_n(x,\beta)$ generated by the following recurrence relations, valid for $n\ge 1$. 
\begin{equation}\label{E:easyrecurrence}
\begin{array}{l}
q_{n+1}=\beta q_n - x q_{n-1}, \ q_0=0, q_1=1; \\ \\  w_{n+1}=\beta w_n - x w_{n-1}, \ w_0=1, w_1=1.
\end{array}
\end{equation}
\end{definition}
Here, $\beta, x\in \mathbb{C}$. The polynomials $q_n(x,\beta)$ generalize the univariate Fibonacci polynomials $F_n(x)=q_n(x,1)$, \cite[p. 327]{FlSe2009}; also $w_n(x,1)=F_{n+1}(x)$.   In the case of steps alone in the classical fair gamblers ruin problem ($a=b=\frac{1}{2}$; $\beta=1$ and $x=\frac{1}{4}z^2$ in \eqref{E:easyrecurrence}), the $\{q_n=F_n(x)\}$ are classically 
\emph{numerator} polynomials, and the $\{w_n=F_{n+1}(x)\}$ are the \emph{denominator} polynomials for the excursion generating function of height less than $n$, namely $K_{n-1}(\frac{1}{2},\frac{1}{2})$ with $r=y=1$ in \eqref{E:KN}; \cite{deBrKnRi1972}; \cite[V.4.3]{FlSe2009}. Here numerator and denominator refer to the convergent of a continued fraction representation of $K_{\infty}$. 
See \cite{Bou-Me2008} for an interesting direction on excursions with different step sets besides the classical steps $\pm 1$.  
\par
We write an \emph{interlacing} property of any two term recurrence $v_{n+1}=\beta v_n-x v_{n-1},$  $n\ge 1$, with coefficients $\beta$ and $x$ independent of $n$: 
\begin{equation}\label{E:vreduction}
\begin{array}{c}
 v_{n+1}v_{n-1}-v_{n}^2=\beta^{-1}x^{n-1}(v_3v_0-v_2v_1), \ \beta\neq 0;
 \end{array}
 \end{equation}
see  \cite[ (2.7)--(2.8)] {Mor2015}. Note that when $v_0=0, \ v_1=1$, the polynomials $v_n=v_n(\beta,-x)$ are called the generalized Fibonacci polynomials in $\beta$ and $-x$, and by standard generating function techniques, the fundamental sequences \eqref{E:easyrecurrence} have closed formulae given as follows: \cite[(2.1) and (2.3)]{Sw1997}; or \cite[(2.11)--(2.12)]{Mor2015}. Define $\alpha:=\sqrt{\beta^2-4x}.$ Then, for all $n\ge 1,$ and with $q_0(x,\beta)=0$, 
\begin{equation}\label{E:qwclosed}
\begin{array}{c}
q_n(x,\beta)= \frac{2^{-n}}{\alpha}\left ((\beta+\alpha)^n-(\beta-\alpha)^n\right ); \\ \\ w_n(x,\beta)= q_n(x,\beta)-xq_{n-1}(x,\beta).
\end{array}
\end{equation} 
The formula for $w_n$ follows from that of $q_n$, for $n\ge 1$, since $q_1-xq_0=1=w_1,$ and $q_2-xq_1=\beta-x=w_2$.
\par
We need some additional notation to describe our method as follows. 
For any pair of integers $m,n\in(-N,N)$ with $m\neq n$ we define the following \emph{first passage} length for the process $\{\mathbf{X}_j\}$ that starts at  $\mathbf{X}_0=m$:
\begin{equation}\label{E:sigmaa}
\mathbf{L}_{m,n}:=\inf\{j\ge 1: \mathbf{X}_j=n \mbox{ or } |\mathbf{X}_j|=N\}.
\end{equation}
For any starting level $\mathbf{X}_0=m$, let $\mathbf{\Gamma}_{m,n}:=\{\mathbf{X}_{j}, j=0,\dots,\mathbf{L}_{m,n}\}$ denote the ordinary first passage path from level $m$ to either level  $n$ or to the boundary of the gambler's ruin process. For our key definition \eqref{E:g}, additional conditions are placed on the first passage path to make it \emph{one--sided}.  
\par
Denote by $\mathbf{R}_{m,n}$ the number of runs and by $\mathbf{V}_{m,n}$ the number of short runs, respectively, along $\mathbf{\Gamma}_{m,n}$, where $\mathbf{L}_{m,n}$ denotes the number of steps along this path. For $n>m$, define $g_{m,n}=g_{m,n}(a,b)$ as the following \emph{upward} conditional joint probability generating function for these counting statistics given two conditions on the path: (1) the path is a one--sided first passage path that starts at $m$ and stays at or above level $m$ until it reaches level $n$, and (2) the first two steps of this path are both in the positive direction. If still $n>m$ then we also define the analogous \emph{downward} conditional joint generating function $g_{n,m}$:
\begin{equation}\label{E:g}
\begin{array}{l}
g_{m,n}:=  E(r^{\mathbf{R}_{m,n}}y^{\mathbf{V}_{m,n}}z^{\mathbf{L}_{m,n}}|\varepsilon_1=\varepsilon_2, \mathbf{X}_0=m, \mathbf{X}_j \ge m, \mbox{ }j=0,\dots,\mathbf{L}_{m,n}). \\ 
g_{n,m}:= E(r^{\mathbf{R}_{n,m}}y^{\mathbf{V}_{n,m}}z^{\mathbf{L}_{n,m}}|\varepsilon_1=\varepsilon_2, \mathbf{X}_0=n, \mathbf{X}_j \le n, \mbox{ }j=0,\dots,\mathbf{L}_{n,m}).
\end{array}
\end{equation}
The condition that the first two steps be in the same direction in the definition \eqref{E:g} arises due to the inclusion of  the statistic  $\mathbf{V}_{m,n} $ in the analysis.  
The path in Figure \ref{F:futuremaxima1} is a downward, first passage path from level 5 to level 0.
\par
Let $n>m$. In the formulation of the recurrence for $g_{m,n}$, we must take account of the unconditional probability that a first passage from level $m$ to level $n$ remains at or above the starting level; we must also define the corresponding probability $\rho_{n,m}$, as follows.
\begin{equation}\label{E:rhogeneral}
\begin{array}{l}
\rho_{m,n}:=P(\mathbf{X}_j\ge m, j=0,\dots,\mathbf{L}_{m,n} | \mathbf{X}_0=m); \\ \rho_{n,m}:=P(\mathbf{X}_j\le n, j=0,\dots,\mathbf{L}_{n,m}| \mathbf{X}_0=n).
\end{array}
\end{equation}
For $a=b=\frac{1}{2}$,
the probability $\rho_{n,0}=\rho_{0,n}$ is determined by the classical solution of the probability of ruin started from fortune $n$ on the interval $[0,n+1]$. For $a=b$, $\rho_{m,n}$ depends only on $k=n-m$ and is determined by $\rho_{m,m+\ell}=\frac{1}{2}(\ell-(\ell-1)a)^{-1}$, \cite[(2.4)]{Moh1955}. 
\par
There are many calculations used to establish various formulae by the help of certain key definitions.
We reserve the phrase \emph{direct calculation} to mean that computer algebra (\emph{Mathematica}, \cite{Mathematica}) is used to help verify the results. A companion document \cite{Mor2018} to the present paper provides details of the verifications. In our approach, the complication of a second stratum is solved by finding the right formulae and then rendering a proof; we often utilize induction based on the proposed formulae. Our proofs may be termed elementary, since we use path decompositions to establish explicit formulae for the conditional generating functions $g_{m,n}$.
\par
Our method for the full model is to show that the appropriate denominators $\{\overline{w}_{m,n}\}$ of the conditional generating functions $g_{m,n}$, together with certain singly--indexed numerators  $\{\overline{q}_n\}$, give rise also to a nice representation of   \eqref{E:KN}; see Theorem \ref{T:KN}, in which our approach involves conditioning on the height $\mathbf{H}=n$ of an excursion. For the homogeneous case of Proposition \ref{P:KN}, the formula for \eqref{E:KN} follows in a standard way of dealing with a finite continued fraction. In the homogeneous case an alternative approach based on the format of  \cite{FlSe2009}, Proposition V.3, could probably be devised. Yet we need  a closed formula for the one--sided first passage generating function $g_{0,N}$ to handle the meander in the full model, and this leads us to take an approach via recurrences proper, not only for $g_{m,n}$ but for $\overline{w}_{m,n}$. Accordingly, by Propositions \ref{P:wformula} and \ref{P:gProp}, we obtain our main results with the help of trigonometric substitutions and direct calculations. 
\section{Proofs of the Building Blocks}\label{S:proofs}\subsection{Recurrence for $g_{m,n}$.}\label{S:gmn}
We first establish the general recurrence relations governing the upward and downward generating functions of \eqref{E:g}. The condition \emph{initial two steps the same} on the trajectory of the lattice path yields immediately that
\begin{equation}\label{E:ginitial}
g_{m,m+2}(a,b):=r z^2, \ g_{m+2,m}(a,b): =rz^2, \ m\ge 0.
\end{equation}
The path decomposition of \cite{Mor2015} handles runs and steps; here we extend that approach for short runs as well.  It is convenient to focus on $g_{n,0}$ with some $n\ge 3$; see the definition \eqref{E:g}. The Figure \ref{F:futuremaxima1} is an illustration of one lattice path counted by $g_{5,0}$. 
Let $U$ or $D$ stand for one step up or down, respectively, in a lattice path, and let $(UD)^{\ell} $ be shorthand for $UDUD\cdots$ with $\ell$ repetitions of the pattern $UD$ for some $\ell \ge 0$.
Since any downward lattice path from $n$ to $0$ must first reach the level $m=1$, 
we have an initial factor $g_{n,1}$ in a product formula for $g_{n,0}$. 
\par
Any section of a lattice path for  $g_{n,1}$, which must end in $DD$, is followed by a sequence of steps of the form $(UD)^{\ell}UU$, or by a \emph{terminal} sequence $(UD)^{\ell}D$. To handle transitions that do not start $UU$ or $DD$ we introduce: 
\begin{equation}\label{E:khomega}
\begin{array}{c}
\omega(a,b):= 1-(1-a)(1-b)r^2y^2z^2, \ k(a,b):=(a+b-ab)/\omega(a,b), \\ 
\tau(a,b):=1+(1-a)(1-b)r^2z^2y(1-y); \ h(a,b):=\tau(a,b)/\omega(a,b).
\end{array}
\end{equation}
Denote $\mathbf{1}=(1,1,1)$ and evaluation of any function $u(a,b)$ at $(r,y,z)$ by $u(a,b)[r,y,z]$. For brevity we may write $u_{a}$ in place of $u(a,a)$. By \eqref{E:khomega}, $k(a,b)[\mathbf{1}]=\tau(a,b)[\mathbf{1}]=1$. Thus $k(a,b)$ is a probability generating function; the term $h(a,b)/h(a,b)[\mathbf{1}]$ is as well. In our discussion of $g_{n,0}$, if $f\ge 3$, then $k_a=k(a,a)$ accounts for a generating factor for an \emph{upward preamble} $(UD)^{\ell}$ from level $m=1$, succeeding $DD$ and preceding $UU$; in this case $k_a=c \sum _{\ell=0}^{\infty} ((1-a)^2r^2y^2z^2)^{\ell}=c/\omega_a$, where $c=a(2-a)$. If instead $f=2$ is the change of stratum parameter, then we obtain $k(a,b)$ in place of $k_a$ due to the fact that now a change in direction at level $m=2$ occurs with probability $(1-b)$ while a change in direction at level $m=1$ occurs with probability $(1-a)$. 
To handle the dependence on $f$, we define
\begin{equation}\label{E:bracketab}
\begin{array}{c}
[a,b]_{m}^{+}:=\left \{ \begin{array}{l} (a,a), \mbox{ if } m\le f-2\\ (a,b),  \mbox{ if } m= f-1 \\ (b,b), \mbox{ if } m\ge f
\end{array}\right \};\ 
[a,b]_{n}^{-}:=\left \{ \begin{array}{l} (a,a), \mbox{ if } n\le f-1\\ (a,b),  \mbox{ if } n= f \\ (b,b), \mbox{ if } n\ge f+1\end{array}\right \}.
\end{array}
\end{equation}
Let us suppose that the continuation of the path after the first downward passage to level $m=1$ is not yet passing into a terminal sequence, so takes the form $(UD)^{k}UU\dots$. 
Starting thus from $UU$ the path makes an upward first passage to level $n$ again (or not), and the pattern ``up to level $n$ and down to level $1$" repeats for an indefinite number of times, $\ell \ge 0$.  To handle the probability associated with the turning of the path downward from a level it will no longer exceed in the future of the path, or in turning from the bottom level $m=1$ to upwards (in the return to level 1), we define the \emph{turning probability at altitude} $m$ by
\begin{equation}\label{E:gamma}
\begin{array}{c}
\gamma_m:=\left \{ \begin{array}{l} 1-a, \mbox{ if } m\le f-1 \\1-b, \mbox{ if } m\ge f \end{array}\right \}.
\end{array}
\end{equation}
By definition \eqref{E:rhogeneral}, it now follows that $g_{n,0}=c g_{n,1} \lambda_{1,n} \lambda_{1,n-1} \cdots \lambda_{1,3} zh[a,b]_1^{+},$ with $\lambda_{1,n}:=\sum_{\ell=0}^{\infty} (4\gamma_1\gamma_n \rho_{1,n}\rho_{n,1}k[a,b]_{1}^{+}k[a,b]_{n}^{-}g_{1,n}g_{n,1})^{\ell},$ or 
$$ \lambda_{1,n}=\frac{1}{1-4\gamma_1\gamma_n \rho_{1,n}\rho_{n,1}k[a,b]_{1}^{+}k[a,b]_{n}^{-}g_{1,n}g_{n,1}}.$$
Here the factor of 4 arises due to the fact that the stationary probabilities for first step up and down, namely $\pi_{+}=\frac{1}{2}$ and $\pi_{-}=\frac{1}{2}$, get replaced by $\gamma_1$ and $\gamma_n$ respectively in $\rho_{1,n}$ and $\rho_{n,1}$.  The factor $k[a,b]_{n}^{-}$ takes account of a \emph{downward  preamble} succeeding $UU$ and preceeding $DD$ from the maximum possible level $M_2\ge 3$ in the remainder of the downward lattice path. Here the successive maximum levels $n=M_1\ge M_2\ge \cdots\ge M_r$ over the whole future of the path, determined in turn from the points of each of its returns to level $m=1$ from the previous such maximum, are the \emph{future maxima} (cf. \cite{Mor2015}) of a downward path from level $n\ge 3$ to level $m=0$. See Figure \ref{F:futuremaxima1}, in which we have  $M_1=5$, and $M_2=4$, $M_3=3$; there is no second future maximum of level 4, for example, because there is no return to level 1 between the two peaks at level 4, but instead we see a downward preamble $(DU)^1$ at $M_2$.  By definition we have  $M_r\ge 3$, and the downward path goes into a \emph{terminal} sequence after a return to level 1 from $M_r$; the terminal sequence is of form $(UD)^{k}D$; see Figure \ref{F:futuremaxima1}.
Eventually, in the beginning, the path will never rise to level $n$ again; but to lower future maxima at levels  $3\le m\le n-1$; thus the product $\lambda_{1,n}\lambda_{1,n-1}\dots \lambda_{1,3}$.  The factor $zh[a,b]_{1}^{+}$ 
corresponds to the terminal sequence. 
Now replace $m=1$ by $m\ge 1$ for a final destination level $m-1$, to obtain the following downward recurrence relation for any $m<n-1$:
\begin{equation}\label{E:G1down}
\begin{array}{c}
g_{n,m-1}=c  z h[a,b]_{m}^{+}g_{n,m} \prod_{j=m+2}^{n} \lambda_{m,j}.
\end{array}
\end{equation}
for a normalization constant $c$ such that $g_{n,m-1}[\mathbf{1}]=1$. Here we officially define $\lambda_{m,j}=\lambda_{m,j}(a,b)$: 
\begin{equation}\label{E:lambdamj}
\lambda_{m,j}:=\frac{1}{1-4\gamma_m\gamma_j \rho_{m,j}\rho_{j,m}k[a,b]_{j}^{-}k[a,b]_{m}^{+}g_{m,j}g_{j,m}}, \ m+2\le j.
\end{equation} 
By symmetric arguments we also obtain the upward recurrence relation for any $m<n-1$: 
\begin{equation}\label{E:G1up}
\begin{array}{c}
g_{m,n+1}= 
c  z h[a,b]_{n}^{-}g_{m,n} \prod_{j=m}^{n-2} \lambda_{j,n},
\end{array}
\end{equation}
where $c$ denotes a generic normalization constant. 
Each factor $\lambda_{m,n}/\lambda_{m,n}[\mathbf{1}]$ defined by \eqref{E:lambdamj} is a probability generating function for a class of paths starting and ending at the same level $n$ (say), with probability of first step given by the turning probability (at $n$). Each path besides the empty path makes a positive number of consecutive down--up transitions of type ``a first passage downward transition to $m$ followed immediately by a first passage upward transition to $n$". See Figure \ref{F:umj}, where a path starts at level $n=3$ at the first marker $\gamma_n$, and makes exactly $2$ down--up transitions between $n=3$ and $m=0$, and ends at the third marker $\gamma_n$. We obtain the same generating function if the paths instead start and end at $m$, with up--down transitions.
\par
By \eqref{E:G1down}--\eqref{E:G1up} we retrieve a closed recurrence for $g_{m,n}$.  Indeed, by \eqref{E:G1up}, for $m<n-2$, 
we simply have $g_{m,n+1}/g_{m+1,n+1}=c_1g_{m,n}\lambda_{m,n}/g_{m+1,n},$ for a normalization constant $c_1$.
Hence,
\begin{equation}\label{E:G2up} 
g_{m,n+1}
=c_1g_{m,n}g_{m+1,n+1}(g_{m+1,n})^{-1}\lambda_{m,n}, \  \ n-m\ge 3.
\end{equation}
Similarly, by applying \eqref{E:G1down}, $g_{n,m-1}/g_{n-1,m-1}=c_2 g_{n,m}\lambda_{m,n}/g_{n-1,m},$
for a normalization constant $c_2$ and $m<n-2$. Hence,
\begin{equation}\label{E:G2down} 
g_{n,m-1}
=c_2g_{n,m}g_{n-1,m-1}(g_{n-1,m})^{-1}\lambda_{m,n}, \ \ n-m\ge 3.
\end{equation}
Observe that the factor $\lambda_{m,n}$ of \eqref{E:lambdamj}, with $m+2<n$, appears exactly the same in both \eqref{E:G1down} and \eqref{E:G1up}, and again in \eqref{E:G2up}--\eqref{E:G2down}. 
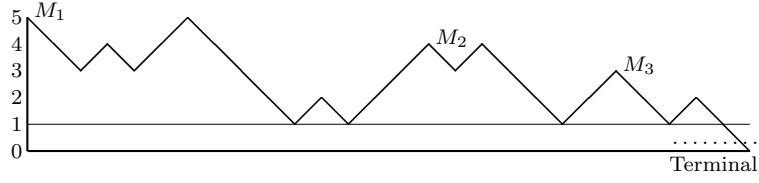
\begin{figure}
\begin{center}
\setlength{\unitlength}{0.0070in}
\begin{picture}(160,100)(190,780)
\thinlines
\put(00,880){\line(1,-1){20}}
\put(-12,775){0}
\put(-12,795){1}
\put(-12,815){2}
\put(-12,835){3}
\put(-12,855){4}
\put(-12,875){5}
\put(05,880){$M_1$}
\put(20,860){\line(1,-1){20}}
\put(40,840){\line(1,1){20}}
\put(60,860){\line(1,-1){20}}
\put(80,840){\line(1,1){20}}
\put(100,860){\line(1,1){20}}
\put(120,880){\line(1,-1){20}}
\put(140,860){\line(1,-1){20}}
\put(160,840){\line(1,-1){20}}
\put(180,820){\line(1,-1){20}}
\put(200,800){\line(1,1){20}}
\put(220,820){\line(1,-1){20}}
\put(240,800){\line(1,1){20}}
\put(260,820){\line(1,1){20}}
\put(280,840){\line(1,1){20}}
\put(300,860){\line(1,-1){20}}
\put(305,860){$M_2$}
\put(320,840){\line(1,1){20}}
\put(340,860){\line(1,-1){20}}
\put(360,840){\line(1,-1){20}}
\put(380,820){\line(1,-1){20}}
\put(400,800){\line(1,1){20}}
\put(445,840){$M_3$}
\put(482,785){\dots\dots\dots}
\put(480,765){Terminal}
\put(420,820){\line(1,1){20}}
\put(440,840){\line(1,-1){20}}
\put(460,820){\line(1,-1){20}}
\put(480,800){\line(1,1){20}}
\put(500,820){\line(1,-1){20}}
\put(520,800){\line(1,-1){20}}
\linethickness{0.050mm}
\put(00,800){\line(1,0){540}}
\thicklines
\put(00,780){\line(1,0){540}}
\put(00,780){\line(0,1){100}}
\end{picture}
\end{center}
\caption{Downward Transition with Future Maxima $M_1=5$, $M_2=4,$ $M_3=3$.} 
\label{F:futuremaxima1}
\end{figure} 
\par
We introduce some notation for the basic method to calculate \eqref{E:KN}, which consists of conditioning on  $\{\mathbf{H}=n\}$. In the remainder of this section we assume $f\ge 3$.
Let $G_{n}$ denote the conditional joint probability generating function of the number of runs, short runs, and steps in an excursion given that the height is $\mathbf{H}=n$ for some $1\le n<N$: 
\begin{equation}\label{E:G}
G_{n}:=E(r^{\mathbf{R}}y^{\mathbf{V}}z^{\mathbf{L}}|\mathbf{H}=n, \mathbf{X}_0=0), \mbox{ \ }n\ge 1.
\end{equation}
In definition \eqref{E:G}, the condition is that after the first step from $m=0$, the path does not return to the $x$-axis until it terminates, but that also, for a positive excursion, the path reaches the specified height, $n$, as a maximum. Now we work with positive excursions. We consider an initial sequence $U(UD)^{\ell}UU$ that brings a lattice path for the first time to level $m=3$ while never returning to level $m=0$. 
The joint generating function for the numbers of runs, short runs, and steps, for only the part $U(UD)^{\ell}$  of this initial sequence is simply $J_a:=a(2-a) zh_a, $
with $h_a$ defined by \eqref{E:khomega}. Now, to make a positive excursion that starts at level $m=0$ and reaches a level $n\ge 3$ for a first time, we also consider any upward path $\Gamma_{1,n}^{+}$  for $g_{1,n}$ that starts at level $m=1$ with $UU$. 
We link the initial sequence $U(UD)^{\ell}UU$ and $\Gamma_{1,n}^{+}$ together by making them overlap on the end $UU$ of the initial sequence and beginning of  $\Gamma_{1,n}^{+}$. Thus the factor of $G_n$ corresponding to a lattice path first reaching level $n\ge 3$ is given by $J_ag_{1,n}$. The remaining factor corresponds to a downward preamble from level $n$ followed by a downward path from level $n$ to level 0.
Hence, 
\begin{equation}\label{E:G3}
G_{n}= a(2-a) z h_{a} g_{1,n}k[a,b]_{n}^{-}g_{n,0}, \ n\ge 3, \ f\ge 3.
\end{equation}
Moreover, by symmetry we have $g_{0,-n}=g_{0,n}$ for all $n\ge 2.$ Hence, the joint generating function of the meander statistics is:
\begin{equation}\label{E:meandergf}
\begin{array}{c}
E\{r^{{\cal R}'_N}y^{{\cal V}'_N}z^{{\cal L}'_N}\}=a(2-a)zh_{a}g_{1,N}, \ \ f\ge 3.
\end{array}
\end{equation}
\subsection{Formula for $\rho_{m,n}$.}\label{S:rho}
In this section we establish a formula for $\rho_{m,n}$ as defined by \eqref{E:rhogeneral}. Note that $1-\rho_{1,N}$ is the probability of ruin  for the gambler's ruin persistence model with two strata on $[0,N]$ in case $\mathbf{X}_0=1$.  The novelty of our approach, based on induction, is unnecessary if $a=b$, since by \cite{Moh1955} a difference equation will solve the probability of ruin in this case.  
\par
The method we use to establish a formula is based first on the future maxima construction of Section \ref{S:gmn}, only  in \eqref{E:rhogeneral} there is no condition on  upward or downward paths starting $UU$ or $DD$. 
In place of $\lambda_{m,j}$ of \eqref{E:lambdamj}, here define:
\begin{equation}\label{E:umj}
u_{m,j}:=\frac{1}{1-4\gamma_m\gamma_j \rho_{m,j}\rho_{j,m}}, \ m+1\le j.
\end{equation}
Let $m<n$. By the way we developed the formulae \eqref{E:G1down}--\eqref{E:G1up}, we have 
\begin{equation}\label{E:rhoupdown}
\begin{array}{l}
(i) \  \rho_{m,n+1}= (1-\gamma_n)\rho_{m,n} \prod\limits_{j=m}^{n-1}u_{j,n}; \ \
(ii) \ \rho_{n,m-1}= (1-\gamma_m)\rho_{n,m} \prod\limits_{j=m+1}^{n}u_{m,j}.
\end{array}
\end{equation}
The factor $(1-\gamma_n)$ in  \eqref{E:rhoupdown}(\emph{i}) gives the probability ($a$ or $b$) of the last step in any one--sided first passage path from level $m$ to level $n$; a similar comment applies to  \eqref{E:rhoupdown}(\emph{ii}). See Figure \ref{F:umj}. By the same method as shown to obtain \eqref{E:G2up}--\eqref{E:G2down}, we have by \eqref{E:umj}--\eqref{E:rhoupdown} that
\begin{equation}\label{E:rhoupdown1}
(i) \ \rho_{m,n+1}
=\frac{\rho_{m,n}\rho_{m+1,n+1}}{\rho_{m+1,n}}u_{m,n}; 
\   (ii) \ \rho_{n,m-1}
=\frac{\rho_{n,m}\rho_{n-1,m-1}}{\rho_{n-1,m}}u_{m,n}.
\end{equation}
\begin{figure}
\begin{center}
\setlength{\unitlength}{0.0070in}
\begin{picture}(160,80)(180,770)
\thinlines
\put(-12,775){0}
\put(-12,795){1}
\put(-12,815){2}
\put(-12,835){3}
\put(-12,855){4}
\put(00,780){\line(1,1){20}}
\put(20,800){\line(1,-1){20}}
\put(40,780){\line(1,1){20}}
\put(60,800){\line(1,1){20}}
\put(80,820){\line(1,1){20}}
\put(100,840){\line(1,-1){20}}
\put(105,845){$\gamma_{n}$}
\put(120,820){\line(1,-1){20}}
\put(140,800){\line(1,1){20}}
\put(160,820){\line(1,1){20}}
\put(180,840){\line(1,-1){20}}
\put(200,820){\line(1,-1){20}}
\put(220,800){\line(1,-1){20}}
\put(245,775){$\gamma_{0}$}
\put(240,780){\line(1,1){20}}
\put(260,800){\line(1,-1){20}}
\put(280,780){\line(1,1){20}}
\put(300,800){\line(1,1){20}}
\put(320,820){\line(1,1){20}}
\put(345,845){$\gamma_{n}$}
\put(340,840){\line(1,-1){20}}
\put(360,820){\line(1,-1){20}}
\put(380,800){\line(1,-1){20}}
\put(400,780){\line(1,1){20}}
\put(405,775){$\gamma_0$}
\put(420,800){\line(1,1){20}}
\put(440,820){\line(1,1){20}}
\put(460,840){\line(1,-1){20}}
\put(465,845){$\gamma_{n}$}
\put(480,820){\line(1,1){20}}
\put(485,815){$\gamma_{2}$}
\put(500,840){\line(1,1){20}}
\put(515,843){$1-\gamma_{n}$}
\linethickness{0.050mm}
\put(00,840){\line(1,0){520}}
\thicklines
\put(00,780){\line(0,1){80}}
\put(00,860){\line(1,0){520}}
\end{picture}
\end{center}
\caption{Illustration of $\rho_{0,n+1}=(1-\gamma_n)\rho_{0,n}\prod\limits_{j=0}^{n-1}u_{j,n}$ for  $n=3$. For the path shown, $u_{0,n}$ has $2$ down--up transitions and $u_{1,n}$ has none, while $u_{2,n}$ has $1$ down--up transition.}
\label{F:umj}
\end{figure}
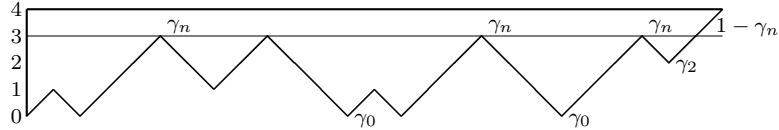 
With the help of \eqref{E:rhoupdown1}, we will now develop a closed recurrence for $\rho_{m,n}$. We first make a definition to establish a convenient form of $\rho_{m,n} $.
\begin{definition}\label{D:Pi}
Let $\rho_{m.n}$ be defined by \eqref{E:rhogeneral}. 
We define a denominator term $\Pi_{m,n}$ for $ \rho_{m,n}$ as follows; $m<n$ in all cases:\\ \\
(I). $\begin{array}{c}  \ \ (1) \ \ \rho_{m,n} =\frac{1}{2}\frac{b/a}{\Pi_{m,n}}, \  m\le f-1; \ \
(2) \ \  \rho_{m,n} =\frac{1}{2}\frac{1}{\Pi_{m,n}}, \    f\le m.
\end{array}$\\ 
\\
(II). $\begin{array}{c} 
 \ \ (1) \ \  \rho_{n,m} =\frac{1}{2}\frac{b/a}{\Pi_{n,m}}, \  n\le f-1; \ \ 
(2) \ \  \rho_{n,m} =\frac{1}{2}\frac{1}{\Pi_{n,m}},\   f\le n.
\end{array}$
\end{definition}
\begin{proposition}\label{P:Pi}
The terms $\Pi_{m,n}$ determined by Definition \ref{D:Pi} satisfy:\\ 
I. Between strata formulae:
\begin{enumerate}
\item $\mbox{ \  } \Pi_{f-\ell,f+j}=j+\ell(\frac{b}{a})-(\ell+j-1)b$,  $\mbox{ \  } \ell\ge 1$, $j\ge 0$; \\
\item $\mbox{ \  } \Pi_{f+j,f-\ell}=(j+1)+(\ell-1)(\frac{b}{a})-(\ell+j-1)b$,  $\mbox{ \  } \ell\ge 1$, $j\ge 0$.
\end{enumerate}
II. Within stratum formulae:
\begin{enumerate}
\item $  \mbox{ \ }\Pi_{m,m+\ell}=\Pi_{m+\ell,m} = \frac{b}{a}\{\ell-(\ell-1)a\}$, $\mbox{ \  } m<m+\ell \le f-1$;  \\
\item $ \mbox{ \  }\Pi_{m,m+j}=\Pi_{m+j,m}=j-(j-1)b$, $\mbox{ \  } f\le m<m+j$.
\end{enumerate}
\end{proposition}
\begin{remark}\label{R:rhohomogeneous} If $a=b$, we have $\Pi_{m,m+\ell}=\Pi_{m+\ell,m}=\ell-(\ell-1)a$ in all cases of Proposition \ref{P:Pi}, consistent with Definition \ref{D:Pi}  and  \cite[(2.4)]{Moh1955}.
\end{remark}
\begin{proof}[Proposition \ref{P:Pi}]  By the Remark \ref{R:rhohomogeneous}, and Definition \ref{D:Pi}, the within stratum formulae \emph{II.1--2} hold in general.  The proof of the between strata cases proceeds by induction on $n-m$, where we assume $n>m$ throughout. Recall that the first step  $\varepsilon_1$ of the gambler's ruin process is determined by $\pi_{+}=P(\varepsilon_{j}=1)=\frac{1}{2}$. We have $\rho_{m,m+1}=\rho_{m,m-1}=\frac{1}{2}$, so the case $n-m=1$ is easily checked. We next verify the cases $n-m=2$ for $\Pi_{m,n}$  and $\Pi_{n,m}$ in \emph{I.1--2}. We apply \eqref{E:umj}--\eqref{E:rhoupdown} with $u_{m,m+1}=(1-\gamma_{m}\gamma_{m+1})^{-1}$. Thus for all $m$, $\rho_{m,m+2}=\frac{1}{2}(1-\gamma_{m+1})/(1-\gamma_{m}\gamma_{m+1})$. In particular, by \eqref{E:gamma}, $\rho_{f-1,f+1}=\frac{1}{2}b/(1-(1-a)(1-b))=\frac{1}{2}(b/a)/\left(1+\frac{b}{a}-b\right)$. 
This gives the correct form for the denominator in \emph{I.1} by Definition  \ref{D:Pi}(I)(1). We apply direct calculation to check the other between strata cases.  Thus all the cases $n-m=2$ have been verified.
\par
Assume by induction that all statements of the proposition hold for $2\le n-m\le k$ for some $k\ge 2$. We wish to show the following induction step: 
\begin{equation}\label{E:claim}
\begin{array}{c}
\mbox{Both  } (i): \ \Pi_{m,n+1}, \mbox{ and } (ii): \ \Pi_{n,m-1}, \mbox{ conform to statements \emph{I.1}} \mbox{ and \emph{I.2}},\\  \mbox{ respectively, } \mbox{ for all }m\le f-1 \mbox{ and } n\ge f, \mbox{ with } n-m=k+1.
 \end{array}
 \end{equation} 
 There are two boundary cases, $\Pi_{f-k-1,f}$ and $\Pi_{f+k,f-1}$, that aren't covered formally by this scheme. However both of these cases actually fall under the within stratum regime. For example, in the calculation of $\rho_{f-k-1,f}$, the one-sided first passage path from $f-k-1$ to $f$ never oscillates between levels $f-1$ and $f$, so the probability $\rho_{f-k-1,f}$ is governed by a single stratum design.  Hence, $\rho_{f-k-1,f}=\frac{1}{2}/(k+1-ka)=\frac{1}{2}(b/a)/\left\{(k+1)\left(\frac{b}{a}\right)-kb\right\}$, consistent with \emph{I.1}. For the other boundary case, by similar reasoning, $\rho_{f+k,f-1}=\frac{1}{2}/(k+1-kb)$, consistent with \emph{I.2}. So the boundary cases have been resolved for all $k$. 
\par
We proceed with our argument for establishing \eqref{E:claim}. By our ranges for $m$ and $n$ we have $\gamma_m=1-a$ and $\gamma_n=1-b$ throughout. By Definition  \ref{D:Pi}(I), we compute $u_{m,n}$ by \eqref{E:umj} under \eqref{E:claim} as follows.
\begin{equation}\label{E:umnformula}
\begin{array}{c}
u_{m,n}=\left\{1-\gamma_m\gamma_n \frac{b/a}{\Pi_{m,n}\Pi_{n,m}}\right \}^{-1}= \frac{\Pi_{m,n}\Pi_{n,m}}{\Pi_{m,n}\Pi_{n,m}-\gamma_m\gamma_n(b/a)}. 
\end{array}
\end{equation}
Now write $m=f-\ell$ and $n=f+j$ for some $\ell\ge 1$ and $j\ge 0$ with $\ell+j=k\ge 2$. By the induction hypothesis we can write $\Pi_{m,n}=j+\ell(\frac{b}{a})-(\ell+j-1)b$, and also $\Pi_{n,m}=(j+1)+(\ell-1)(\frac{b}{a})-(\ell+j-1)b$. Now, by direct calculation, we have a simple identity for the denominator of the right hand side of \eqref{E:umnformula}: 
\begin{equation}\label{E:PimnPinmID}
\begin{array}{c}
\Pi_{m,n}\Pi_{n,m}-\gamma_m\gamma_n(b/a)=\left \{j+1+\ell(\frac{b}{a})-(\ell+j)b\right \}\Pi_{m+1,n},
\end{array}
\end{equation}
where we applied the induction hypothesis for $\Pi_{m,n}$, $\Pi_{n,m}$, and $\Pi_{m+1,n}$.
Now rewrite \eqref{E:rhoupdown1}(\emph{i}) by applying Definition \ref{D:Pi} and   \eqref{E:umnformula}--\eqref{E:PimnPinmID}, as follows. We have that $\rho_{m,n+1}$ is given by:
\begin{equation}\label{E:rhoup2}
\begin{array}{c}
\left[\frac{\frac{1}{2}b/a}{\Pi_{m,n}}\right] \left[\frac{\frac{1}{2}b/a}{\Pi_{m+1,n+1}}\right ] \left[\frac{\frac{1}{2}b/a}{\Pi_{m+1,n}}\right ]^{-1}\frac{\Pi_{m,n}\Pi_{n,m}}{\left \{j+1+\ell(\frac{b}{a})-(\ell+j)b\right \}\Pi_{m+1,n}},
\end{array}
\end{equation}
with the caveat that if $m+1=f$, then the factor $b/a$ in the second and third factors on the left is replaced by 1. Finally we use that, by the induction hypothesis and all statements of the proposition themselves, we have
$\Pi_{m+1,n+1}=\Pi_{n,m}$ for all $n>m$ under \eqref{E:claim}. Therefore, simply by cancellation of 3 $\Pi$--factors, \eqref{E:rhoup2} yields
$ \rho_{m,n+1}
= \frac{\frac{1}{2}b/a}{\left \{j+1+\ell(\frac{b}{a})-(\ell+j)b\right \} }.$
Thus by Definition  \ref{D:Pi}(I)(1), the induction step \eqref{E:claim} has been verified for case (\emph{i}). The argument for the downward case (\emph{ii}) is wholly similar to the upward case (\emph{i}). In fact by direct calculation the relevant identity in place of \eqref{E:PimnPinmID} is the same except with $\Pi_{n-1,m},$ in place of $\Pi_{m+1,n}$.  And in \eqref{E:rhoup2} the roles of $b/a$ and 1 are reversed. Thus  $\rho_{f+j,f-\ell-1}= \frac{1}{2}/\left \{j+1+\ell(\frac{b}{a})-(\ell+j)b\right \}$, as required. Therefore the induction step \eqref{E:claim} has been verified.  
\qed 
\end{proof} 
\subsection{The denominators $\overline{w}_{m,n}$ of $g_{m,n}$.}\label{S:wbar}
We first consider the  homogeneous case $a=b$, and establish formulae for the denominators $w_n^{*}(a)$ of $g_{0,n}(a,a)$ defined by \eqref{E:g}, where of course $g_{m,n}(a,a)$ depends only on $|n-m|$. We will abbreviate $g_n:=g_{0,n}$ without confusion for this homogeneous case. 
Denote $\rho_n:=\rho_{0,n}=\frac{1}{2}/(n-(n-1)a)$, by Definition \ref{D:Pi} and Proposition \ref{P:Pi}, and $\lambda_n:=\lambda_{m,m+n}=\{1-4(1-a)^2k_a^{2}\rho_n^{2}g_n^{2}\}^{-1}$, defined by \eqref{E:lambdamj}.   By \eqref{E:G2up}, we have 
\begin{equation}\label{E:gup} 
g_{n+1}
=c_1g_{n}^{2}g_{n-1}^{-1}\lambda_{n}, \ \  n \ge 3; \ \ g_3=czh_a g_2 \lambda_2.
\end{equation}
We now establish that, for a certain sequence of polynomials $\{w_n^{*}(a)=w_n^{*}(a;r,y,z), n\ge 1\}$, with constant coefficient 1, we have
\begin{equation}\label{E:ghomogeneous}
g_{n}=C_{n,a}\omega_ar z^n \tau_{a}^{n-2}/w_n^{*}(a),   n\ge 2;  \  C_{n,a}:=a^{n-2}(n-(n-1)a)/(2-a).
\end{equation} 
The proposed formula \eqref{E:ghomogeneous} holds for $n=2$ with $w_{2}^{*}(a):=\omega_a$, since $C_{2,a}=1$. We also define $w_{1}^{*}(a):=1$. 
Motivated by the idea that $w_n^{*}(a)$ satisfies a Fibonacci recurrence, we introduce
\begin{equation}\label{E:tauxbetahomogeneous}
\begin{array}{c}
x_a:=a^2z^2\tau_{a}^2; \ \ \beta_a: = 1+z^2(a^2-(1-a)^2r^2(y^2+a^2(1-y)^2z^2)),
\end{array}
\end{equation}
and we define  
\begin{equation}\label{E:wnstara}
w_{n+1}^{*}(a)=\beta_aw_n^{*}(a)-x_aw_{n-1}^{*}(a), \ n\ge 2.
\end{equation}
The form of \eqref{E:tauxbetahomogeneous} used to make the definition \eqref{E:wnstara} may be guessed by taking account of \eqref{E:vreduction} together with the proposed form \eqref{E:ghomogeneous}.
That is, we already have defined $w_1^{*}(a)$ and $w_2^{*}(a)$, consistent with \eqref{E:ghomogeneous}, and we can derive $g_3$ via \eqref{E:gup}. So we will have thereby guessed $w_3^{*}(a)$.  We can likewise predict $w_4^{*}(a)$. But  \eqref{E:vreduction} gives that the appropriate $x_a$  for \eqref{E:wnstara} is $x_a=\left(w_{3}^{*}(a)^2-w_4^{*}(a)w_2^{*}(a)\right)/\left(w_{2}^{*}(a)^2-w_3^{*}(a)w_1^{*}(a)\right)$. Once we have $x_a$, we find $\beta_a$ via \eqref{E:wnstara}, and we also extend the definition \eqref{E:wnstara} to $n=0$ by solving \ref{E:wnstara} backwards:
$w_0^{*}(a):=(\beta_a-\omega_a)/x_a.$
We define also the associated numerators $\{q_n^{*}(a)\}$ defined by the Fibonacci recurrence $q_{n+1}^{*}(a)=\beta_aq_n^{*}(a)-x_aq_{n-1}^{*}(a), \ n\ge 1,$ with initial conditions
\begin{equation}\label{E:qwstar0}
\begin{array}{c}
q_0^{*}(a):=-(1-y)(1+y+(1-a)^2r^2y^2z^2(1-y))/\tau_a^2, \ \  q_1^{*}(a):=y^2.
\end{array}
\end{equation}
By the choice of $q_1^{*}(a)$ we obtain the form $K_1=r^2z^2q_1^{*}(a)/w_1^{*}(a)$ for \eqref{E:KN}. By the choice of $q_0^{*}(a)$ we obtain by direct computation an interlacing form $w_{n+1}^{*}q_{n}^{*}-w_{n}^{*}q_{n+1}^{*}=a^2z^2x_a^{n-1}$ at $n=0$.
By direct computation to  check the initial conditions for Fibonacci recurrences, we have:
\begin{equation}\label{E:qwstar1}
\begin{array}{c}
q_n^{*}(a)=c_1q_n(x_a,\beta_a) +c_2w_n(x_a,\beta_a),\ \ c_2:= q_0^{*}(a), c_1= y^2 - c_2; \\ 
w_n^{*}(a)=c'_1q_n(x_a,\beta_a) +c'_2w_n(x_a,\beta_a),  \  \ c'_2:= w_0^{*}(a), c'_1= 1 - c'_2.
\end{array}
\end{equation}
\par
We first verify \eqref{E:ghomogeneous} for $n=3$. 
By \eqref{E:khomega} and \eqref{E:G1up}, we have $g_3=czh_ag_2\lambda_2=cz(\tau_a/\omega_a)rz^2(1-a^2(1-a)^2r^2z^4/\omega_a^2)^{-1}$, since $\rho_2=\frac{1}{2}/(2-a)$ and $k_a=a(2-a)/\omega_a$. This yields $g_3=C_{3,a} \omega_a rz^3 \tau_a /w_{3}^{*}(a)$, by direct computation.
Now assume by induction that \eqref{E:ghomogeneous} holds with $m$ in place of $n$ for all $2\le m\le n$, for some $n\ge 3$. Then by  \eqref{E:gup} we have $g_{n+1}=c_{n+1}[\omega_ar z^n \tau_{a}^{n-2}/w_n^{*}]^2[\omega_a r z^{n-1} \tau_{a}^{n-3}/w_{n-1}^{*}]^{-1}\lambda_n$, where $c_{n+1}$ incorporates both the constant $c_1$ of \eqref{E:gup} and the factor $C_{n,a}^2/C_{n-1,a}$. By direct substitution of the induction hypothesis, and taking care to write the term  $g_n^{2}$ that appears in $\lambda_n$ in terms of $x_a$  via $\tau_{a}^2=a^{-2}z^{-2}x_a$, so that $\lambda_n=\left(1-a^2(1-a)^2r^2z^4x_{a}^{n-2}/w_n^{*}(a)^2\right)^{-1},$ we obtain 
\begin{equation}\label{E:ginduction}
g_{n+1}=c_{n+1}\omega_ar z^{n+1} \tau_{a}^{n-1}w_{n-1}^{*}(a)/\{w_n^{*}(a)^{2}-a^2(1-a)^2r^2z^4x_a^{n-2}\}. 
\end{equation} 
 To compute the denominator in this last expression, we note the following.
\begin{lemma}\label{L:homogeneousinterlacing}
Let $w_{n}^{*}(a)$  be defined as the solution to \eqref{E:wnstara}. Then for all $n\ge 1$ we have: \ \ 
$w_{n}^{*}(a)^{2}-w_{n+1}^{*}(a)w_{n-1}^{*}(a)=a^2(1-a)^2r^2z^4x_a^{n-2}.$
\end{lemma}
\begin{proof} By the definition \eqref{E:wnstara} and by \eqref{E:vreduction} we have:
\begin{equation}\label{E:homogeneousinterlacing}
w_{n}^{*}(a)^{2}-w_{n+1}^{*}(a)w_{n-1}^{*}(a)=-\beta_{a}^{-1}x_{a}^{n-1}(w_{3}^{*}(a)w_{0}^{*}(a)-w_{2}^{*}(a)w_{1}^{*}(a)).
\end{equation}
By direct calculation, $w_{3}^{*}(a)w_{0}^{*}(a)-w_{2}^{*}(a)w_{1}^{*}(a)=-a^2r^2z^4(1-a)^2\beta_{a}/x_{a}.$ Hence the lemma follows by substitution of this last formula into \eqref{E:homogeneousinterlacing}. 
\qed \end{proof}
Up to the form of the constant $C_{n,a}$, relation \eqref{E:ghomogeneous} now follows by induction from \eqref{E:ginduction} and Lemma \ref{L:homogeneousinterlacing}. To verify the constant, we need only verify the claim: $w_n^{*}(a)[\mathbf{1}]=a^{n-1}(n-(n-1)a)$. This is easily verified by induction, \eqref{E:wnstara}, and direct computation. Hence we have verified \eqref{E:ghomogeneous}.
\par
Now turn to the full model. We recursively define an array of functions $\{\overline{w}_{m,n}=\overline{w}_{m,n}(a,b)\}$ such that $\overline{w}_{m,n}$ will turn out to be the denominator polynomial with constant term 1 for the rational expression of $g_{m,n}$. We first define initial cases:
\begin{equation}\label{E:wbarIC}
\begin{array}{l}
\overline{w}_{m,m+2}:=\omega[a,b]_{m}^{+},  \ m\ge0;\ \ \overline{w}_{n,n-2}:=\omega[a,b]_{n}^{-},  \ n-2\ge0; \\ \overline{w}_{m,m+1}=\overline{w}_{m+1,m}=1, \ m\ge 0.
\end{array}
\end{equation}
For example, if $m\le f-2$, then $[a,b]_{m}^{+}=(a,a)$, so $\overline{w}_{m,m+2}:=\omega_a$. 
We require a generalization of $x_a,$ and $\beta_a$  of \eqref{E:tauxbetahomogeneous} to make our definition of $\{\overline{w}_{m,n}\}$  for two strata, as follows. Define
\begin{equation}\label{E:tauxbeta}
\begin{array}{c}
x(a,b):=b^2z^2\tau^2(a,b); \ \ \beta(a,b)=\beta_b-(b-a)b^2(1-b)r^2(1-y)^2z^4.
\end{array}
\end{equation}
for $\beta_b$ defined by \eqref{E:tauxbetahomogeneous}.
Here we note that $\tau(a,b)$  is symmetric in $a$ and $b$, so $x(b,a)=a^2z^2\tau^2(a,b)$ for $\tau(a,b)$  defined by \eqref{E:khomega}. 
\begin{definition} \label{D:wbar}  
Denote $\overline{w}_{m,n}=\overline{w}_{m,n}(a,b)$.\\
(I) Define the upward denominator  $\overline{w}_{m,n}$ for all $n-m\ge 2$ by:\\
$\begin{array}{l} 
(1) \mbox{ \  } \overline{w}_{m,m+\ell}:=w_{\ell}^{*}(a)$,  $\mbox{ \  }m< m+\ell \le f$;  $\mbox{ \  } \overline{w}_{m,m+\ell}:=w_{\ell}^{*}(b),  \mbox{ \  }f\le m<m+\ell;
\\ (2)\mbox{ \ \ \ \  } \overline{w}_{f-\ell,f+1}:=\frac{1-b}{1-a}w_{\ell+1}^{*}(a)+\frac{b-a}{1-a}w_{\ell}^{*}(a),  \mbox{ \  } 1\le \ell \le f;
\\ (3)\mbox{ \ \ \ \ } \overline{w}_{m,f+2}:=\beta(a,b)\overline{w}_{m,f+1}-x(a,b)\overline{w}_{m,f},  \mbox{ \  } m\le f-1;
\\(4)\mbox{ \ \ \ \  }  \overline{w}_{m,f+j+1}:=\beta_b\overline{w}_{m,f+j}-x_b\overline{w}_{m,f+j-1},  \mbox{ \  } m\le f-1, \  j\ge 2.
\end{array}$\\ \\
(II) Define the downward denominator $\overline{w}_{n,m}$ for all $n-m\ge 2$ by:\\
$\begin{array}{l} 
(1) \mbox{ \  } \overline{w}_{m+\ell,m}:=w_{\ell}^{*}(a), \mbox{ \  }m<m+\ell \le f-1;  \mbox{ \  } \overline{w}_{m+\ell,m}:=w_{\ell}^{*}(b),  \mbox{ \  }f-1\le m;
\\ (2)\mbox{ \ \ \ \  }\overline{w}_{f+j,f-2}:=\frac{1-a}{1-b}w_{j+2}^{*}(b)+\frac{a-b}{1-b}w_{j+1}^{*}(b),  \mbox{ \  } 0\le j;
\\ (3)\mbox{ \ \ \ \  } \overline{w}_{n,f-3}:=\beta(b,a)\overline{w}_{n,f-2}-x(b,a)\overline{w}_{n,f-1},  \mbox{ \  } f\le n;
\\ (4)\mbox{ \ \ \ \  } \overline{w}_{n,f-\ell-2}:=\beta_a\overline{w}_{n,f-\ell-1}-x_a\overline{w}_{n,f-\ell},  \mbox{ \  } f\le n, \ \ell\ge 2. 
\end{array}$
\end{definition}
Notice that in Definition \ref{D:wbar}(II), we are effectively reversing the roles of $a$ and $b$ from (I). 
In case $a=b$, we simply have $\overline{w}_{m,n}=w_{|n-m|}^{*}(a), \ |n-m|\ge 2$. 
We write the first step of \emph{crossing over the threshold of the stratum} in either upward or downward directions as a linear combination of two successive homogeneous case solutions. For the next step over the threshold we use the \emph{mixed} parameters for $x$ and $\beta$, and for further steps we use the appropriate homogeneous parameters for $x$ and $\beta$. With no crossing over a stratum, the homogeneous solution is shown. Finally, Definition \ref{D:wbar} and \eqref{E:wbarIC} are consistent. For example, in part (I)(2) of the Definition,  we find: 
$\begin{array}{c}
\overline{w}_{f-1,f+1}=\frac{1-b}{1-a}w_{2}^{*}(a)+\frac{b-a}{1-a}w_{1}^{*}(a) =\frac{1-b}{1-a}\omega(a,a)+\frac{b-a}{1-a}=\omega(a,b)
\end{array}$. 
\subsubsection{Interlacing identity and closed formula for $\overline{w}_{m,n}$.}\label{S:interlacing}
To establish a formula for $g_{m,n}$, we will employ an interlacing identity for the denominators  $\overline{w}_{m,n}$. 
Define the interlacing bracket:  
\begin{equation}\label{E:bracketwbar}
[\overline{w}]_{m,n}:=\overline{w}_{m,n}\overline{w}_{m+1,n+1}-\overline{w}_{m,n+1}\overline{w}_{m+1,n}, \ m\le n-2.
\end{equation}
It actually suffices to consider only the upward direction for  $[\overline{w}]_{m,n}$, since by Lemma \ref{L:wbarID}, the natural corresponding downward definition, $ [\overline{w}]_{n,m}:=  \overline{w}_{n,m}\overline{w}_{n-1,m-1}-\overline{w}_{n,m-1}\overline{w}_{n-1,m}, \ m\le n-2,$
satisfies $[\overline{w}]_{n,m}=[\overline{w}]_{m,n}.$
\begin{proposition}\label{P:interlacing}
The following identities hold for $[\overline{w}]_{m,n}$: 
\begin{enumerate}
\item $\mbox{  }  [\overline{w}]_{f-\ell,f+j}=a^2r^2z^4(1-a)(1-b)x_a^{\ell-2}x(a,b)x_b^{j-1},$  $\mbox{ \  } \ell\ge 2, \ j\ge1$;\\
\item $\mbox{  }  [\overline{w}]_{f-\ell,f}=a^2r^2z^4(1-a)(1-b)x_a^{\ell-2}$,  $\mbox{ \  } \ell\ge 2$; \\
\item $\mbox{   }  [\overline{w}]_{f-1,f+j}=b^2r^2z^4(1-a)(1-b)x_b^{j-1},$  $\mbox{ \  } j\ge1$;\\
\item $ [\overline{w}]_{m,m+\ell}=a^2r^2z^4(1-a)^2x_a^{\ell-2}$, $\mbox{ } m+\ell\le f-1;$  \\
\item $[\overline{w}]_{m,m+j}=b^2r^2z^4(1-b)^2x_b^{j-2};$  $\mbox{ } f\le m$, \ $j \ge 2$.
\end{enumerate}
\end{proposition}
\begin{proof} By Definition \ref{D:wbar}(I)(1) and by Lemma \ref{L:homogeneousinterlacing} we have that statements \emph{4--5} of the proposition hold. Next fix $\ell\ge 2$ and notice that the case $j=0$ in \emph{1} is similar to the case of statement \emph{2}, the difference being $x(a,b)\neq x_b$. We will first verify \emph{2}. 
Thus we write, using the Definition \ref{D:wbar} and \eqref{E:bracketwbar}, that $ [\overline{w}]_{f-\ell,f}$ is given by:
\begin{equation}\label{E:interlacing1}
\begin{array}{l}
 w_{\ell}^{*}(a)\{\frac{1-b}{1-a}w_{\ell}^{*}(a)+\frac{b-a}{1-a}w_{\ell-1}^{*}(a)\}-\{\frac{1-b}{1-a}w_{\ell+1}^{*}(a)+\frac{b-a}{1-a}w_{\ell}^{*}(a)\}w_{\ell-1}^{*}(a).
 \end{array}
 \end{equation}
 The $w_{\ell}^{*}(a)w_{\ell-1}^{*}(a)$ terms cancel in \eqref{E:interlacing1}. Thus by \eqref{E:interlacing1} and Lemma \ref{L:homogeneousinterlacing}, 
$[\overline{w}]_{f-\ell,f}=\frac{1-b}{1-a}\left ( w_{\ell}^{*}(a)^2-w_{\ell+1}^{*}(a)w_{\ell-1}^{*}(a)\right)= a^2(1-a)(1-b)r^2z^4x_{a}^{\ell-2}.$
 Hence statement \emph{2} is proved. 
 \par
 We now turn to statement \emph{1}. Fix $\ell\ge 2$ and  let $j\ge 0$. Denote $[a,b]_0=(a,b)$ and $[a,b]_j=(b,b)$ for $j\ge1$. Thus, by Definition \ref{D:wbar}(I)(3)--(4)  and \eqref{E:bracketwbar},
 \begin{equation}\label{E:interlacing3}
\begin{array}{c}
 [\overline{w}]_{f-\ell,f+j+1}=\overline{w}_{f-\ell,f+j+1}\left \{\beta[a,b]_{j}\overline{w}_{f-\ell+1,f+j+1}-x[a,b]_{j}\overline{w}_{f-\ell+1,f+j}\right \}\\ \\ -\left \{\beta[a,b]_{j}\overline{w}_{f-\ell,f+j+1}-x[a,b]_{j}\overline{w}_{f-\ell,f+j}\right \} \overline{w}_{f-\ell+1,f+j+1}.
 \end{array}
 \end{equation}
 Now the terms of \eqref{E:interlacing3} involving $\beta[a,b]_j$ cancel and we obtain from \eqref{E:interlacing3} and \eqref{E:bracketwbar} that
 \begin{equation}\label{E:interlacing4}
 [\overline{w}]_{f-\ell,f+j+1}=x[a,b]_{j} [\overline{w}]_{f-\ell,f+j}.
 \end{equation}
 Now put $j=0$ in \eqref{E:interlacing4} and conclude by \emph{2} and \eqref{E:interlacing4} that statement \emph{1} holds for the initial case $j=1$ for the given fixed $\ell\ge 2$. Now for the same fixed index $\ell$, take statement \emph{1} as an induction hypothesis for induction on $j\ge 1$. We have just established this induction hypothesis for $j=1$. Thus verify by \eqref{E:interlacing4} again that the induction step holds since $x[a,b]_j=x(b,b)=x_b$ for all $j\ge 1$.  Thus statement \emph{1} is proved.  
 \par
Finally we turn to statement \emph{3}. We note that \eqref{E:interlacing3}--\eqref{E:interlacing4} continues to hold by Definition \ref{D:wbar}(I)(3) with $\ell=1$ as long as $j\ge 1$. Now we compute by \eqref{E:wbarIC}--\eqref{E:tauxbeta},  Definition \ref{D:wbar}(I)(3), and the interlacing bracket definition \eqref{E:bracketwbar} that, since by  \eqref{E:wbarIC}, $\overline{w}_{f-1,f+1} = \omega(a,b)$, while by Definition \ref{D:wbar}, $\overline{w}_{f,f+2} =w_2^{*}(b)= \omega(b,b)$, 
 \begin{equation}\label{E:interlacing5}
 \begin{array}{l}
 [\overline{w}]_{f-1,f+1}=\omega(a,b)\omega(b,b)-\left(\beta(a,b)\omega(a,b)-x(a,b)\cdot 1\right)\cdot 1\\ 
 \ \ \ \ \ \ \ \ \ \ \ \ \ \  =\omega(a,b)\left(\omega(b,b)-\beta(a,b)\right)+x(a,b)= b^2(1-a)(1-b)r^2z^4,
 \end{array}
 \end{equation}
 where at the last step we make a direct calculation based on the definitions in \eqref{E:khomega} and \eqref{E:tauxbeta}.
 Now take statement \emph{3} as an induction hypothesis for induction on $j\ge 1$. By \eqref{E:interlacing5}  have established this induction hypothesis for $j=1$. Thus verify by \eqref{E:interlacing4} with $\ell=1$ and $j\ge 1$ that the induction step holds since $x[a,b]_j=x(b,b)=x_b$ for all $j\ge 1$.  So, statement \emph{3} is proved. \qed \end{proof}
\par
We turn to the task of obtaining a closed formula for $\overline{w}_{m,n}$.
By Definition \ref{D:wbar}(I)(4), given $m=f-\ell<f$, $\overline{w}_{m,f+1}$  and $\overline{w}_{m,f+2}$ form the initial conditions for a recurrence $\overline{w}_{m,f+j+1}:=\beta_b\overline{w}_{m,f+j}-x_b\overline{w}_{m,f+j-1}$, $j\ge 2$. Put $m=f-\ell$ for some $\ell\ge 1$. We denote the vector of these upward initial conditions across the stratum threshold by the $2\times1$ vector $\mathbf{W}(\ell)$. Then we define a $2\times2$ matrix $Q(b)$, and for each $\ell<f$, a $2\times1$ vector $\mathbf{d}=\mathbf{d}(\ell)$ by   
\begin{equation}\label{E:Q}
Q(b):=
\left[ 
\begin{array}{ll}
q_1^{*}(b) & w_1^{*}(b) \\ 
q_2^{*}(b)& w_2^{*}(b)
\end{array}\right] , \ \mathbf{W}(\ell):=
\left[ \begin{array}{l} \overline{w}_{f-\ell,f+1} \\ 
\overline{w}_{f-\ell,f+2} \end{array}\right ] =Q(b)\mathbf{d}; \ 
\mathbf{d}:=\left[ \begin{array}{l}d_1(\ell) \\ d_2(\ell) \end{array} \right ].
\end{equation}
By Definitions \ref{D:wbar}(I)(1--3), 
we can write each term of the right side of the recurrence of (I)(3) using (I)(1--2) in terms of $w_{\ell}^{*}(a)$ and $w_{\ell+1}^{*}(a)$ as follows:
$\overline{w}_{f-\ell,f+1}=\frac{1-b}{1-a}w_{\ell+1}^{*}(a)+\frac{b-a}{1-a}w_{\ell}^{*}(a)$ and $\overline{w}_{f-\ell,f}=w_{\ell}^{*}(a)$, so 
$\overline{w}_{f-\ell,f+2}
=\beta(a,b)\left(\frac{1-b}{1-a}w_{\ell+1}^{*}(a)+\frac{b-a}{1-a}w_{\ell}^{*}(a)\right )-x(a,b)w_{\ell}^{*}(a).$ We combine terms with the notation
$\kappa(a,b):=\left(\frac{b-a}{1-a}\right)\beta(a,b)-x(a,b).$ Thus 
\begin{equation}\label{E:B3}
\mathbf{W}(\ell) =B
\left[ \begin{array}{l}w_{\ell}^{*}(a)\\ w_{\ell+1}^{*}(a)
\end{array}\right ]; \ \ \ B=\left[ 
\begin{array}{ll}
\ \ \frac{b-a}{1-a}&\ \ \ \ \ \frac{1-b}{1-a}\\ 
\kappa(a,b)& \ \frac{1-b}{1-a}\beta(a,b)
\end{array}
\right ].
\end{equation}
By equating the two expressions for the vector $\mathbf{W}(\ell)$ in \eqref{E:Q} and \eqref{E:B3}, we recover 
\begin{equation}\label{E:M}
\mathbf{d}(\ell) =\left[ \begin{array}{l}d_1(\ell)\\ d_2(\ell)\end{array}\right ]=
M \left[ \begin{array}{l}w_{\ell}^{*}(a)\\ 
w_{\ell+1}^{*}(a)\end{array}\right ]; \ \ M:=Q(b)^{-1}B.
\end{equation}
Here it is clear that the entries of the matrix $M=\left (\mu_{i,j}\right)$, with $\mu_{i,j}=\mu_{i,j}(a,b)$ $1\le i,j\le 2$, do not depend on $\ell$. 
We note by direct calculation from \eqref{E:Q} that $\det \left(Q(b)\right)=-b^2z^2$, so we have a straightforward formula for $M$ via \eqref{E:Q} and \eqref{E:M}.
\begin{proposition}\label{P:wformula} Let $d_1(\ell)$ and $d_2(\ell)$ be defined by \eqref{E:Q}--\eqref{E:M}. Then
\begin{equation}\label{E:wformulaup}
\overline{w}_{f-\ell,f+j}=d_1(\ell)q_j^{*}(b)+d_2(\ell)w_j^{*}(b), \ \ell \ge 1, \ j\ge 1.
\end{equation}
\end{proposition}
\begin{proof}
Fix $\ell \ge 1$. By Definition \ref{D:wbar}(I)(4), for all $j\ge 2$ it holds that $\overline{w}_{f-\ell,f+j+1}=\beta_b\overline{w}_{f-\ell,f+j}-x_b\overline{w}_{f-\ell,f+j-1}$. But if we denote the right side of \eqref{E:wformulaup} by $v_j$, then also $v_{j+1}=\beta_bv_{j}-x_bv_{j-1}$, $j\ge 2$, because by construction each of $\{q_j^{*}(b)\} $ and $\{w_j^{*}(b) \}$ satisfy the same two term recurrence, and the coefficients $d_1(\ell)$ and $d_2(\ell)$ in \eqref{E:wformulaup} are independent of $j$. Also by definition \eqref{E:Q}, for any given $\ell\ge 1$, \eqref{E:wformulaup} holds for $j=1$ and $j=2$, that is, $v_j=\overline{w}_{f-\ell,f+j}$, $j=1,2$. Hence we have $v_j=\overline{w}_{f-\ell,f+j}$ for all $j\ge 1$. Since $\ell$ was arbitrary the proof is complete. \qed  
\end{proof} 
\begin{lemma}\label{L:wbarID}
For all $1\le m<n$, there holds: \  
$\overline{w}_{m,n}=\overline{w}_{n-1,m-1}.$
\end{lemma}
\begin{proof}
Notice that the  lemma holds in the initial cases $n-m=1,2$ by \eqref{E:wbarIC}. Also, if  $f\le m<n$ or $1\le m<n \le f$ then the statement holds by (I)(1) and (II)(1) in Definition~\ref{D:wbar}. So consider now $\overline{w}_{f-\ell,f+j}$ for $1\le \ell <f$ and $j\ge 1$. Our method is to prove that the statement:
$\mbox{(H)}_{\ell,j} \ \  \overline{w}_{f-\ell,f+j}=\overline{w}_{f+j-1,f-\ell-1},$
holds for both the initial cases $\ell =1$ and $\ell=2$, and all $ j\ge 1.$ 
\par  
We first establish $\mbox{(H)}_{\ell,j}$ for $\ell=1$ and all $j\ge 1$. On the one hand, write $\overline{w}_{f-1,f+j}$ by \eqref{E:wformulaup}  with $\ell=1$, and on the other hand, write $\overline{w}_{f+j-1,f-2}$  by Definition 
\ref{D:wbar}(II)(2), as follows.
\begin{equation}\label{E:wbarIDbasecase}
\begin{array}{c}
\overline{w}_{f-1,f+j}=d_1(1)q_j^{*}(b)+d_2(1)w_j^{*}(b); \\ \overline{w}_{f+j-1,f-2}=\frac{1-a}{1-b}w_{j+1}^{*}(b)+\frac{a-b}{1-b}w_j^{*}(b).
\end{array}
\end{equation}
By \eqref{E:Q}, \eqref{E:M} and direct calculation, we have that $d_1(1)=\mu_{1,1}w_1^{*}(a)+\mu_{1,2}w_2^{*}(a)=-(1-a)(1-b)r^2z^2,$ and $d_2(1)=\mu_{2,1}w_1^{*}(a)+\mu_{2,2}w_2^{*}(a)=1$.
Therefore, by substitution into \eqref{E:wbarIDbasecase}, we find that the two expressions in \eqref{E:wbarIDbasecase} are equal if and only if 
\begin{equation}\label{E:Hinitial}
  -(1-b)^2r^2z^2q_j^{*}(b)=w_{j+1}^{*}(b)-w_j^{*}(b).
\end{equation}
By direct computation we check that \eqref{E:Hinitial} is true at both $j=1$ and $j=2$. Thus since $\{q_j^{*}(b)\}$ and $\{w_j^{*}(b)\}$ each satisfy the same Fibonacci recurrence, \eqref{E:Hinitial} holds for all $j\ge 1$.
\par
Next we establish that $\mbox{(H)}_{\ell,j}$ holds with $\ell=2$ and all $j\ge 1$. 
Write $\overline{w}_{f-2,f+j}$ by \eqref{E:wformulaup}  with $\ell=2$, and write $\overline{w}_{f+j-1,f-3}$  by Definition \ref{D:wbar}(II)(3), as follows. 
\begin{equation}\label{E:wbarIDbasecase2}
\begin{array}{l}
\mbox{(i)} \ \ \  \overline{w}_{f-2,f+j}=d_1(2)q_j^{*}(b)+d_2(2)w_j^{*}(b); \\\mbox{(ii)} \ \  \overline{w}_{f+j-1,f-3}=\beta(b,a)\overline{w}_{f+j-1,f-2}-x(b,a)\overline{w}_{f+j-1,f-1},
\end{array}
\end{equation}
with $  \overline{w}_{f+j-1,f-2} = \frac{1-a}{1-b}w_{j+1}^{*}(b)+\frac{a-b}{1-b}w_{j}^{*}(b); \  \ \overline{w}_{f+j-1,f-1}=w_{j}^{*}(b).$ By \eqref{E:Q} and \eqref{E:M} we directly verify that
$d_1(2)= -(1-a)(1-b)r^2z^2\beta(b,a);  \  d_2(2) = \beta(b,a)-x(b,a).$
To verify that the  expressions (i) and (ii) in \eqref{E:wbarIDbasecase2} are equal, we substitute $d_1(2)$ and $d_2(2)$, and obtain, after a little algebra in which $x(b,a)x_j^{*}(b)$ cancels on the two sides,  the condition: 
$ -(1-b)^2r^2z^2\beta(b,a)q_j^{*}(b)=\beta(b,a)\left(w_{j+1}^{*}(b)-w_j^{*}(b)\right), \ \forall \ j\ge 1.$ 
This is obviously equivalent to the condition \eqref{E:Hinitial}. Hence the two expressions in \eqref{E:wbarIDbasecase2} are equal for all $j\ge 1$,  so $\mbox{(H)}_{\ell,j}$ holds also at $\ell=2$ for all $j\ge 1$.
\par
Finally, fix any $j\ge1$. We appeal to \eqref{E:M} and \eqref{E:wformulaup} and  to Definition \ref{D:wbar}(II)(4), to obtain, for any $\ell\ge 3$,
\begin{equation}\label{E:wbarIDbasecase3}
\begin{array}{l}
\overline{w}_{f-\ell,f+j}= \left(\mu_{1,1}w_{\ell}^{*}(a)+\mu_{1,2}w_{\ell+1}^{*}(a)\right)q_j^{*}(b)+\left(\mu_{2,1}w_{\ell}^{*}(a)+\mu_{2,2}w_{\ell+1}^{*}(a)\right)w_j^{*}(b);\\  \overline{w}_{f+j-1,f-\ell-1}=\beta_a \overline{w}_{f+j-1,f-\ell}-x_a \overline{w}_{f+j-1,f-\ell+1}.
\end{array}
\end{equation} 
For any  $\ell\ge 1$, write $u_{\ell}:=\overline{w}_{f-\ell,f+j}$ and $v_{\ell}:=\overline{w}_{f+j-1,f-\ell-1}$ for the two lines of \eqref{E:wbarIDbasecase3}. Since $u_{\ell}$ is a linear combination of two successive terms of the sequence  $\{w_{\ell}^{*}(a)\}$, it follows that,$\{u_{\ell}\}$ itself satisfies the recursion $u_{\ell+1}=\beta_au_{\ell}-x_au_{\ell-1}, \ell\ge 2$. But also $\{v_{\ell}\}$ satisfies the same recurrence. Moreover, we proved that $\mbox{(H)}_{\ell,j}$ holds for $\ell=1$ and $\ell=2$, so we have $u_{1}=v_{1}$, and $u_2=v_2$. Therefore we have $u_{\ell}=v_{\ell}$ for all $\ell\ge 1$.  Thus by \eqref{E:wbarIDbasecase3}, $\mbox{(H)}_{\ell,j}$ is proved for all $\ell\ge 1$. Since $j\ge 1$ was arbitrary, $\mbox{(H)}_{\ell,j}$ is true  $ \forall\ \ell, j\ge 1$. \qed
\end{proof}
\begin{lemma}\label{L:C}
The following identities hold. \
\begin{enumerate}
\item $ \  \overline{w}_{f-\ell,f+j}[\mathbf{1}]=
a^{\ell}b^{j-1}\Pi_{f-\ell,f+j},\  \forall \  \ell\ge 1, \ j\ge 1$. \\
\item $  \   \overline{w}_{f+j,f-\ell}[\mathbf{1}]=
a^{\ell-1}b^{j}\Pi_{f+j,f-\ell},  \mbox{  } \forall \  \ell\ge 2,\ j\ge 0$.\\
\item $ \ q_{\ell}^{*}(a)[\mathbf{1}]=\ell a^{\ell-1},$ 
$ \ \  w_{\ell}^{*}(a)[\mathbf{1}]=
a^{\ell-1}\left( \ell - (\ell -1)a\right);  \ \forall \  \ell \ge 1$.
\end{enumerate}
\end{lemma}
\begin{proof} At $(r,y,z)=\mathbf{1}$ we have $\beta_a=2a$ and $x_a=a^2$. Thus $\alpha=0$ in \eqref{E:qwclosed}. 
Therefore by \eqref{E:qwclosed}, $q_{\ell}^{*}(a)[\mathbf{1}]=\lim_{\alpha\to 0}\frac{2^{-\ell}}{\alpha}\{(2a+\alpha)^{\ell}-(2a-\alpha)^{\ell}\}=\ell a^{\ell-1}$. Thus, by the second formula of  \eqref{E:qwclosed}, we obtain $w_{\ell}^{*}(a)[\mathbf{1}]$  by $x_a[\mathbf{1}]=a^2$,  
so \emph{3} is proved. Now apply \eqref{E:wformulaup}, also at $(r,y,z)=\mathbf{1}$. By \eqref{E:M} and direct calculation, 
$d_1(\ell)[\mathbf{1}]=-(1-a)(1-b)\ell a^{\ell-1}$, and $d_2(\ell)[\mathbf{1}]=
a^{\ell-1}[\ell-(\ell-1)a]$. Now plug in $q_{j}^{*}(b)[\mathbf{1}]$
and $w_{j}^{*}(b)[\mathbf{1}]$ 
from \emph{3}, into \eqref{E:wformulaup} to obtain formula \emph{1}  from Proposition \ref{P:wformula} after direct simplification. 
The proof of \emph{2} follows from  \emph{1} and Lemma \ref{L:wbarID}, in view of Definition \ref{D:Pi}. 
\qed 
\end{proof}
\subsection{Closed formula for $g_{m,n}$.}\label{S:gformula}
\begin{proposition}\label{P:gProp} 
We have the following formulae for $g_{m,n}$. \\ 
I. The formulae for upward between--strata cases, $j\ge 1$ and $\ell\ge 2$:
\begin{enumerate}
\item $\mbox{ } g_{f-\ell,f+j}=\frac{\omega(a,a)}{2-a}r z^{j+\ell}\tau(a,b)[a\tau(a,a)]^{\ell-2}[b\tau(b,b)]^{j-1}\left(a\Pi_{f-\ell,f+j}/\overline{w}_{f-\ell,f+j}\right)$, 
\item $\mbox{ }g_{f-1,f+j}=\frac{\omega(a,b)}{a+b-ab} r z^{j+1}[b\tau(b,b)]^{j-1}\left(a\Pi_{f-1,f+j}/ \overline{w}_{f-1,f+j}\right)$;
\end{enumerate}
II. The formulae for downward between--strata cases, $j\ge 1$ and $\ell\ge 2$:
\begin{enumerate}
\item $\mbox{ } g_{f+j,f-\ell}=\frac{\omega(b,b)}{2-b} r z^{j+\ell} \tau(a,b)[a\tau(a,a)]^{\ell-2}[b\tau(b,b)]^{j-1}\left(a\Pi_{f+j,f-\ell}/\overline{w}_{f+j,f-\ell}\right)$,  
\item $\mbox{ } g_{f,f-\ell}=\frac{\omega(a,b)}{a+b-ab }r z^{\ell}[a\tau(a,a)]^{\ell-2}\left(a\Pi_{f,f-\ell}/\overline{w}_{f,f-\ell}\right)$;
\end{enumerate}
III. The formulae for within stratum cases:
\begin{enumerate}
\item  $\mbox{ }g_{m,m+\ell}=g_{m+\ell,m}=\frac{\omega(a,a)}{2-a} r z^{\ell}[a\tau(a,a)]^{\ell-2}\left(\frac{a}{b}\Pi_{m,m+\ell}/w_{\ell}^{*}(a)\right),\\\mbox{ } m<m+\ell\le f-1$; 
\begin{enumerate} 
\renewcommand{\labelenumii}{\theenumii}
\renewcommand{\theenumii}{\theenumi.\arabic{enumii}}
\item $\mbox{  }g_{f-\ell,f}=\frac{\omega(a,a)}{2-a} r z^{\ell} [a\tau(a,a)]^{\ell-2} \left(\frac{a}{b}\Pi_{f-\ell,f}/w_{\ell}^{*}(a)\right) $, $\ell\ge 1$;
\end{enumerate}
\item $\mbox{ }g_{m,m+j}=g_{m+j,m}=\frac{\omega(b,b)}{2-b} r z^{j}[b\tau(b,b)]^{j-2}\left(\Pi_{m,m+j}/w_{j}^{*}(b)\right) , \\ \mbox{ } f \le m<m+j$; 
\begin{enumerate} 
\renewcommand{\labelenumii}{\theenumii}
\renewcommand{\theenumii}{\theenumi.\arabic{enumii}}
\item $\mbox{ }g_{f+j,f-1}=\frac{\omega(b,b)}{2-b} r z^{j+1} [b\tau(b,b)]^{j-1} \left(\Pi_{f+j,f-1}/w_{j+1}^{*}(b)\right) $, $j\ge 1$.
\end{enumerate}
\end{enumerate}
Furthermore, the following identity holds for all $n\ge m+2$, where $\lambda_{m,n}$ is defined by \eqref{E:lambdamj}.
\begin{equation}\label{E:interlacinglambda}
 \lambda_{m,n}
=\frac{\overline{w}_{m,n}\overline{w}_{m+1,n+1}}{\overline{w}_{m,n+1}\overline{w}_{m+1,n}} .
\end{equation}
\end{proposition}
\begin{remark}\label{R:gconstant}
Since $\tau(a,b)[\mathbf{1}]=1$ for all $a$ and $b$, one easily checks by  Definition \ref{D:Pi} and Lemma \ref{L:C} that the formulae of Proposition \ref{P:gProp} all yield the evaluation $g_{m,n}[\mathbf{1}]=1$. The factor $\Pi_{m,n}$ appears in Lemma \ref{L:C} the same as it does in the statements \emph{I-II}, so these factors cancel at $\mathbf{1}$.  
\end{remark}
\begin{remark}\label{R:ghomogeneous}
All formulae in \emph{III} hold by \eqref{E:ghomogeneous} for the homogeneous case. For example, in the statement \emph{III.1}, we have  $\frac{a}{b}\Pi_{m,m+\ell}=\ell-(\ell-1)a$, so there is no dependence on $b$. 
\end{remark}
\begin{proof}[Proposition \ref{P:gProp}] Recall by \eqref{E:ginitial} that  $g_{m,m+2}=g_{m+2,m}=rz^2$. One can easily check by Definitions \ref{D:Pi} and \eqref{E:wbarIC} that in each of \emph{I.1} with $j=1$, and  \emph{II.2} with $\ell =2$, the formulae reduce to $rz^2$.
By \eqref{E:G2up}--\eqref{E:G2down}, we must calculate a term $\lambda_{m,n}$ defined by \eqref{E:lambdamj}. The term $\lambda_{m,n}$ is the same in both \eqref{E:G2up} and \eqref{E:G2down}, so we only consider $m<n$ in \eqref{E:lambdamj}. The structure of the proof is to first establish \eqref{E:interlacinglambda} for $n=m+2$ and to establish the initial cases $n-m=3$ of the statements \emph{I--II} of the proposition.
Following this, an induction step will be established for all cases at once, wherein an inductive step for \eqref{E:interlacinglambda} shall be the main stepping stone of the proof.
\par
Thus consider first $n:=m+2$ in \eqref{E:lambdamj}. We consider 4 cases: (i) $n\le f-1$; (ii) $ m=f-2,\ n=f$; (iii) $m=f-1,\  n=f+1$; (iv) $m\ge f$. We verify by \eqref{E:khomega}--\eqref{E:gamma}, Definition \ref{D:Pi}, Proposition \ref{P:Pi},  Proposition \ref{P:interlacing}, and direct calculation, that in all cases (i)--(iv),  $\ \lambda_{m,n}=\frac{\omega[a,b]_{m}^{+}\omega[a,b]_{m+1}^{+}}{\omega[a,b]_{m}^{+}\omega[a,b]_{m+1}^{+} - [\overline{w}]_{m,n}} $. Verification of this initial identity by direct calculation suffices for \eqref{E:interlacinglambda}, since for the numerator we have by   \eqref{E:wbarIC} that $\overline{w}_{m,n}=\omega[a,b]_{m}^{+}$ and $\overline{w}_{m+1,n+1}=\omega[a,b]_{m+1}^{+}$, and since for the denominator we have by Definition \eqref{E:bracketwbar} that
$\overline{w}_{m,n}\overline{w}_{m+1,n+1}-[\overline{w}]_{m,n}=\overline{w}_{m+1,n}\overline{w}_{m,n+1}$.
\par
We turn to the initial conditions for \emph{I--II}. There are again four cases to consider. We conform with the notation of \eqref{E:G1down} and \eqref{E:G1up}. For the upward cases we write the lower index $m$ and the upper index $m+3$. For the downward cases we write the upper index $m+2$ and the lower index $m-1$. The cases are (I.1) $m=f-2, \ m+3= f+1$; (I.2) $m=f-1, \ m+3= f+2$; (II.1) $m+2=f+1, \ m-1= f-2$; (II.2) $m+2=f, \ m-1= f-3$.
We use direct calculation of $g_{m,m+3}$ or $g_{m+2,m-1}$ for the upward and downward cases, respectively. Besides the formulae  \eqref{E:G1down} and \eqref{E:G1up}, we use  $\lambda_{m,m+2}$  given by \eqref{E:interlacinglambda}, where $\overline{w}_{m+1,m+2}=1$ by definition \eqref{E:wbarIC}. Since the denominator of $\lambda_{m,m+2}$ in each case is $\overline{w}_{m,m+3}=\overline{w}_{m+2,m-1}$ by Lemma \ref{L:wbarID}, we  compute $p_{m,m+3}:=(1/c)g_{m,m+3}\overline{w}_{m,m+3}$ and $p_{m+2,m-1}:=(1/c)g_{m+2,m-1}\overline{w}_{m,m+3}$ in the upwards and downwards cases respectively. 
Schematically, since $g[\mathbf{1}]=1$, we have $p/p[\mathbf{1}]=g\overline{w} /\overline{w} [\mathbf{1}]=numerator/\overline{w} [\mathbf{1}]$, where $numerator$ stands for the stated formula without the denominator $\overline{w}$. By cancellation of the $\Pi$--factors as in Remark \ref{R:gconstant}, 
we match $p/p[\mathbf{1}]$ with $(numerator/\Pi)/(\overline{w}[\mathbf{1}]/\Pi)$ for verification by direct calculation.
\par
We now proceed by induction on all cases of the proposition at once, where we assume that all statements hold for $g_{m,n}$ and $g_{n,m}$ with $2\le n-m\le  k$, for some $k\ge 3$. By the above we have established all the initial cases, $k=3$, for this hypothesis; as noted earlier, the case $n-m=2$ is trivial. We now apply the formulae \eqref{E:G2up} and \eqref{E:G2down} to establish an induction step in each of the upward  \emph{I.1--2} and downward \emph{II.1--2} cases respectively. We are allowed to use any of the statements of \emph{III} by Remark \ref{R:ghomogeneous}. Notice that for the range of indices we must now consider, in all cases $n\ge f$ and $m\le f-1$, so by \eqref{E:gamma}, $\gamma_m\gamma_n=(1-a)(1-b)$.
\par
Consider first \emph{I.1}. Let first (i) $n+1=m+k+1$, for some $m\le f-2$ and $n\ge f+1$; there is another subcase (ii) $n=f$, that we handle as a special case by direct calculation below.  We rewrite  \eqref{E:G2up} for easy reference: 
$ g_{m,n+1}
=c_1g_{m,n}g_{m+1,n+1}(g_{m+1,n})^{-1}\lambda_{m,n}.$
In the definition \eqref{E:lambdamj} we have by \eqref{E:khomega}--\eqref{E:gamma} that 
$k[a,b]_m^{+}=k(a,a)=\frac{a(2-a)}{\omega(a,a)}$, $k[a,b]_n^{-}=k(b,b)=\frac{b(2-b)}{\omega(b,b)}$.
Also, by Definition \ref{D:Pi} and Proposition \ref{P:Pi}, 
$4 \rho_{m,n}\rho_{n,m}=\frac{(b/a)}{\Pi_{m,n}\Pi_{n,m}}= \frac{ab}{\left(a\Pi_{m,n}\right)\left(a\Pi_{n,m}\right)}.$
Therefore by  \eqref{E:lambdamj} and the induction hypothesis \emph{I.1}, for $g_{m,n}$, and \emph{II.1}, for $g_{n,m}$, the expression $(\dagger) \ 1-1/ \lambda_{m.n}$, is written:
\begin{equation}\label{E:lambdainduction}
\begin{array}{l}
\frac{\gamma_m\gamma_n ab }{\left(a\Pi_{m,n}\right)\left(a\Pi_{n,m}\right)}\frac{ab(2-a)(2-b)}{\omega(a,a)\omega(b,b)}g_{m,n}g_{n,m}=
 \frac{\gamma_m \gamma_n a^2r^2 z^{2j+2\ell}b^2\tau^2(a,b)[a^2\tau^{2}_{a}]^{\ell-2}[b^2\tau^{2}_{b}]^{j-1}}{\overline{w}_{m,n}\overline{w}_{n,m}}. 
 \end{array}
\end{equation}
Now apply \eqref{E:tauxbetahomogeneous} and \eqref{E:tauxbeta} to write $x_a$, $x(a,b)=b^2z^2\tau^2(a,b)$, and $x_b$, using all but 4 powers of $z$. So the numerator of the right member of  \eqref{E:lambdainduction} is simply the interlacing bracket $[\overline{w}]_{m,n}$ of Proposition \ref{P:interlacing}.\emph{1}. Thus, after writing  $\overline{w}_{n,m} =\overline{w}_{m+1,n+1}$ by Lemma \ref{L:wbarID}, and applying the bracket definition \eqref{E:bracketwbar}, we have established that $(\dagger)$ is given by $\frac{\overline{w}_{m,n}\overline{w}_{m+1,n+1}-\overline{w}_{m,n+1}\overline{w}_{m+1,n}}{\overline{w}_{m,n}\overline{w}_{m+1,n+1}}$, so   \eqref{E:interlacinglambda} holds. 
Finally, apply  \eqref{E:G2up} and the induction hypothesis \emph{I.1}  and \eqref{E:interlacinglambda}. Since the lower index $m+1$ is the same in both the numerator and denominator of the ratio $g_{m+1,n+1}/g_{m+1,n}$, we obtain, by  \emph{I.1}  for $m+1\le f-2$, or by \emph{I.2} for $m+1=f-1$, that  $g_{m+1,n+1}/g_{m+1,n}=cz [b\tau(b,b)]\overline{w}_{m+1,n}/\overline{w}_{m+1,n+1}.$ Thus,  $g_{m,n+1}=c_1\frac{g_{m,n}g_{m+1,n+1}}{g_{m+1,n}}\frac{\overline{w}_{m,n}\overline{w}_{m+1,n+1}}{\overline{w}_{m,n+1}\overline{w}_{m+1,n}}=cz[b\tau(b,b)]\left(g_{m,n}\overline{w}_{m,n}\right)/\overline{w}_{m,n+1}$. Hence by plugging in the numerator $p_{m,n}:= g_{m,n}\overline{w}_{m,n}$ (ignoring the constants) from the induction hypothesis for \emph{I.1}, the induction step for \emph{I.1}(i), including  \eqref{E:interlacinglambda}, is complete by Remark \ref{R:gconstant}. 
\par
To recapitulate, in general, there are two steps, where for  the upward and downward cases we conform to the recurrences \eqref{E:G2up} and \eqref{E:G2down}, respectively.
\begin{enumerate}
\item Establish \eqref{E:interlacinglambda} by showing that  the numerator in the analogue of the right hand member of \eqref{E:lambdainduction} gives a bracket $[\overline{w}_{m,n}] \ ( \ =[\overline{w}_{n,m}]\ )$ from Proposition \ref{P:interlacing}, for the parameters $m,n$ of $\lambda_{m,n}$. 
\item Establish that when the induction hypothesis is applied, the condition  $(u) \ \frac{p_{m,n}p_{m+1,n+1}}{p_{m,n}[\mathbf{1}]p_{m+1,n+1}[\mathbf{1}]}-\frac{p_{m+1,n}p_{m,n+1}}{p_{m+1,n}[\mathbf{1}]p_{m,n+1}[\mathbf{1}]}=0,$ is verified for \emph{I}, and condition
$(d) \ \frac{p_{n,m}p_{n-1,m-1}}{p_{n,m}[\mathbf{1}]p_{n-1,m-1}[\mathbf{1}]}-\frac{p_{n-1,m}p_{n,m-1}}{p_{n-1,m}[\mathbf{1}]p_{n,m-1}[\mathbf{1}]}=0,$ is verified for \emph{II}. 
\end{enumerate}
For all the remaining cases of the induction steps in \emph{I--II}, including the subcase \emph{I.1}(ii), we proceed by direct calculation to check the details of the these 2 Steps.  
In Step 1 it is implicit that the factors of $\omega(a,b)$ that occur variously by substitution from factors $k(a,b)$ in the formula for $\lambda_{m,n}$, and also from the numerators of  $g_{m,n}$ and $g_{n,m}$, cancel one another in every case. This is borne out in the direct calculations, where the pattern of substitutions from the induction hypothesis is shown.
By Remark \ref{R:gconstant}, conditions (u)--(d) are equivalent to showing, for the ratio $\frac{p_{m+1,n+1}}{p_{m+1,n}}=cz\tau $, in the upward case, or $\frac{p_{n-1,m-1}}{p_{n-1,m}}=cz\tau$, in the downward case, that the factor of $z\tau$ completes the form of the numerator $p_{m,n+1}$ [resp. $p_{n,m-1}$] as one extra factor of the numerator form $p_{m,n}$ [resp. $p_{n,m}$].  Here the factor $\tau$ depends on subcases; it is $\tau(a,b)$ in subcases \emph{I.1}(ii), and in \emph{II.1}(ii): $n\ge f+2,\ m=f-1$. We show the pattern of substitutions for (u)--(d) in the direct calculations, \cite{Mor2018}. 
\qed 
\end{proof}
\subsection{Generating function of the excursion statistics.}\label{S:proofKN}
 We derive a closed formula for $K_N(a,b)$ of \eqref{E:KN}. 
Recall by \eqref{E:qwstar1} 
that $\{ q_n^{*}(a)\}$ and  $\{w_n^{*}(a)\}$ share  a common Fibonacci recurrence: $v_{n+1}=\beta_a v_n-x_a v_{n-1}$, $n\ge 1$. We extend the $\{q_{n}^{*}(a)\}$ from the homogeneous model to the full model as $\{\overline{q}_n\}$, analogous to $\{\overline{w}_{0,n}\}$ of Definition \ref{D:wbar}, except with $\overline{q}_n$ we start the stratum crossing at index $n=f$ rather than $n=f+1$.
\begin{definition} \label{D:qbar}
Define $\overline{q}_{n}=\overline{q}_{n}$ for all $n\ge 1$ by:\\ \\
$\begin{array}{l}
\mbox{(1)} \ \   \overline{q}_{n}:=q_{n}^{*}(a), \ 1\le n<f; \ \ \ \ \ \ \ \ \ \ \   \mbox{(2)} \ \  \overline{q}_{f}:=\frac{1-b}{1-a}q_{f}^{*}(a)+\frac{b-a}{1-a}q_{f-1}^{*}(a);\\ \\
\mbox{(3)} \ \overline{q}_{f+1}:=\beta(a,b)\overline{q}_{f}-x(a,b)\overline{q}_{f-1}; \ \mbox{(4)} \  \overline{q}_{f+j+1}:=\beta_b\overline{q}_{f+j}-x_b\overline{q}_{f+j-1},   j\ge1.
\end{array}$
\end{definition}
Denote the single indexed bracket (cf. Casorati determinant, \cite{Is2005})
\begin{equation}\label{E:bracketqw}
[ \overline{w},\overline{q} ]_n := \overline{w}_{n,0} \overline{q}_{n+1} - \overline{q}_{n} \overline{w}_{n+1,0}, \  n\ge 1;
\end{equation}
We note that the homogeneous case of \eqref{E:bracketqw} is written 
  $[w^{*}(a),q^{*}(a)] _{n} :=w_{n}^{*} (a) q_{n+1}^{*} (a) - q_{n}^{*} (a) w_{n+1}^{*} (a),\  n\ge 1.$
 \begin{lemma}\label{L:bracketqw}
The following identities hold:
 \begin{enumerate}
\item  $\ \  [w^{*}(a),q^{*}(a)]_n = a^2z^2x_a^{n-1},\ n\ge 1;$\\
\item $ \ \ [ \overline{w},\overline{q} ]_{f-1}=\frac{1-b}{1-a}[w^{*}(a),q^{*}(a)] _{f-1} =\frac{1-b}{1-a}a^2z^2x_a^{f-2} ;$\\
\item $\ \ [ \overline{w},\overline{q} ]_{f+j-1}=\frac{1-b}{1-a}a^2z^2x_a^{f-2} x(a,b)x_b^{j-1}, \ \forall \  j\ge 1.$
\end{enumerate}
\end{lemma}
Before we can prove Lemma \ref{L:bracketqw} we write a formula for $\overline{q}_n$ as follows.
\begin{lemma}\label{L:Qformula}
Let  $M:=Q(b)^{-1}B$, for $Q(b)$ defined by \eqref{E:Q} and $B$ defined by \eqref{E:B3}. Then, 
 \begin{equation}\label{E:qbarmatrixform}
 \overline{q}_{f+j-1}= \left[ \begin{array}{ll}
q_{j}^{*}(b) & w_{j}^{*} (b)
\end{array}\right ] M\left[ \begin{array}{l}
q_{f-1}^{*}(a) \\ q_{f}^{*} (a)\end{array}\right ], \  \forall \ j \ge 1.
\end{equation}
\end{lemma}
\begin{proof} The proof is almost the same as the proof of Proposition \ref{P:wformula}. Define $Q(b)$ as before in \eqref{E:Q}, but write now a revision $\mathbf{W}_q(f)$ of $\mathbf{W}(\ell)$, and write  $\mathbf{W}_q(f)$ in two ways 
 as follows.
\begin{equation}\label{E:Qprime}
 \mathbf{W}_q(f):=\left[ \begin{array}{l}\overline{q}_{f} \\ \overline{q}_{f+1}\end{array}\right ] =Q(b)\mathbf{d}_q(f); \ \ \ \mathbf{W}_q(f) =B\left[ \begin{array}{l}q_{f-1}^{*}(a) \\ \ q_{f}^{*}(a) \end{array}\right ].
 \end{equation}
We have $B$ given by \eqref{E:B3}, because by Definitions \ref{D:wbar} and \ref{D:qbar} the equations that define $B$ in \eqref {E:Qprime} are the same as those defining $B$ in \eqref{E:B3}.
By equating the two expressions for the vector $\mathbf{W}_q(f)$ in \eqref{E:Qprime}, we  recover 
$\mathbf{d}_q(f) =M \left[ \begin{array}{l}q_{f-1}^{*}(a)\\ \ q_{f}^{*}(a)\end{array}\right ].$
Now the formula \eqref{E:qbarmatrixform} follows because by \eqref{E:Qprime} and the definition of $Q(b)$ we have established the formula for $j=1,2$. Therefore by the fact that either side  of \eqref{E:qbarmatrixform}  satisfies the same recurrence $v_{j+1}=\beta_b v_j-x_b v_{j-1}$, $j\ge 2$, we have that both sides are equal as stated. \qed \end{proof}
\begin{proof}[Lemma \ref{L:bracketqw}.] We first prove statement \emph{1}. By the simple fact that $\{q_n^{*}(a)\}$ and $\{w_n^{*}(a)\} $ satisfy the same Fibonacci recurrence, we have:
$\begin{array}{l} [w^{*}(a),q^{*}(a)]_n=  w_n^{*}(\beta_aq_n^{*}- x_a q_{n-1}^{*})-q_n^{*}(\beta_aw_n^{*}-x_a w_{n-1}^{*})
 =x_a[w^{*},q^{*}]_{n-1} \end{array}$, \\
holds for all $n\ge 1$, where we suppressed the $a$ in $q_n^{*}$ and $w_n^{*}$. By direct calculation from \eqref{E:tauxbetahomogeneous}--\eqref{E:qwstar0} and \eqref{E:bracketqw},
$[w^{*}(a),q^{*}(a)]_0=w_0^{*}(a)q_{1}^{*}(a)-q_0^{*}(a)w_{1}^{*}(a)=a^2z^2/x_a.$
Therefore, since we may iterate the one--step recursion for  $[w^{*}(a),q^{*}(a)]_n$, we obtain statement \emph{1} of the lemma. 
\par
Next,  by Defintions \ref{D:wbar} and \ref{D:qbar}, $[\overline{w},\overline{q}]_{f-1}=\overline{w}_{1,f}\overline{q}_f-\overline{q}_{f-1}\overline{w}_{1,f+1}$ 
$=\begin{array}{l} w_{f-1}^{*}\left( \frac{1-b}{1-a}q_{f}^{*}+\frac{b-a}{1-a}q_{f-1}^{*}\right )-q_{f-1}^{*}\left( \frac{1-b}{1-a}w_{f}^{*}+\frac{b-a}{1-a}w_{f-1}^{*}\right )  =\frac{1-b}{1-a}[w^{*},q^{*}]_{f-1}.\end{array}$
Therefore statement \emph{2} of the lemma follows by statement \emph{1}.
\par
Finally, note that by Lemma \ref{L:wbarID} we have $\overline{w}_{n,0}=\overline{w}_{1,n+1}$. Further by  \eqref{E:M} and Proposition \ref{P:wformula} with $\ell = f-1$, we may write 
 \begin{equation}\label{E:wn0formula}            
 \begin{array}{l}
\ \ \overline{w}_{1,f+j}= \left[ \begin{array}{ll} q_{j}^{*}(b) & w_{j}^{*} (b)\end{array}\right ] 
M \left[  \begin{array}{l} w_{f-1}^{*}(a) \\ w_{f}^{*} (a) \end{array}\right ].
\end{array}
\end{equation}
Now write $n=f+j-1$ for some $j\ge 1$. Also for simplicity abbreviate $q_{0}=q_{f-1}^{*}(a)$,  $q_{1}=q_{f}^{*}(a)$, $w_{0}=w_{f-1}^{*}(a)$, $w_{1}=w_{f}^{*}(a)$, $u=q_{j}^{*}(b)$, $v=w_{j}^{*}(b)$, $U=q_{j+1}^{*}(b)$, $V=w_{j+1}^{*}(b)$. By \eqref{E:bracketqw}, Lemma \ref{L:Qformula}, and \eqref{E:wn0formula}, 
$$[\overline{w},\overline{q}]_n= \left[ \begin{array}{ll}u& v \end{array}\right ] M
\left\{ \left[  \begin{array}{l} w_{0}\\ w_{1} \end{array}\right ]\left[ \begin{array}{ll}U & V \end{array}
\right ] M
\left[  \begin{array}{l} q_{0}\\ q_{1} \end{array} \right ] 
 - \left[  \begin{array}{l} q_{0} \\ q_{1} \end{array} \right ]\left[ \begin{array}{ll} U& V \end{array}\right ] 
 M
\left[ \begin{array}{l}w_{0} \\ w_{1}\end{array}\right ] \right \}.$$
Denote $M=(\mu_{i,j})$, and calculate the expression under the curly brackets as follows:
$ (w_0q_1-q_0w_1) \left [ \begin{array}{ll} \ \ \mu_{1,2}& \ \ \mu_{2,2} \\ -\mu_{1,1} &-\mu_{2,1} \end{array} \right ] \left [ \begin{array}{ll} U\\ V \end{array} \right ].$
Therefore, after multiplying through by $\left[ \begin{array}{ll}u& v \end{array}\right ] M$, we obtain:\\
$[\overline{w},\overline{q}]_n=(w_0q_1-q_0w_1)\left[ \begin{array}{ll}u& v \end{array}\right ] \left[  \begin{array}{ll} \ \ \ \  0 & \det(M) \\ -\det(M) & \ \  0\end{array}\right ]\left[ \begin{array}{ll}U\\ V\end{array}\right ].$
Putting back our variables we thus have 
$$ [\overline{w},\overline{q}]_n= [ w^{*}(a),q^{*}(a) ]_{f-1} (-\det(M))[ w^{*}(b),q^{*}(b)]_j, \mbox{ valid for all } j\ge 1.$$ Also, by direct calculation,  $\det(M)= -\frac{1-b}{1-a}\tau(a,b)^2$. 
Thus by \eqref{E:tauxbeta} and statement \emph{1}, the statement \emph{3} of the lemma is proved. 
\qed 
\end{proof}
\begin{theorem}\label{T:KN}
The conditional generating function \eqref{E:KN}  has the following formula.
\[K_{N}(a,b)=C_{N,a,b}\frac{r^2z ^2\overline{q}_{N}}{\overline{w}_{1,N+1}},\ N\ge 1,\]
where $\overline{q}_{N}$ and $\overline{w}_{1,N+1}$ are given by Lemma \ref{L:Qformula} and \eqref{E:wn0formula} respectively, each with $j=N-f+1$, and where $C_{N,a,b}=(1-a)\frac{(N-f)a+(f-1)b-(N-1)ab}{(N-f)a+(f-1)b-Nab}.$
\end{theorem}
In the homogeneous case $a=b$, Theorem \ref{T:KN} takes the following form.
\begin{proposition}\label{P:KN}
Suppose $a=b$. Then $K_N(a)$ of \eqref{E:KN}  has the following formula.
$K_{N}(a)=\frac{1}{N}(N-(N-1)a)\frac{r^2z ^2q_{N}^{*}(a)}{w_{N}^{*}(a)},\ N\ge 1,$
where $q_{N}^{*}(a)$ and $w_{N}^{*}(a)$ are defined by \eqref{E:tauxbetahomogeneous}--\eqref{E:qwstar1}. 
\end{proposition}
\begin{proof}[Proposition \ref{P:KN} and Theorem \ref{T:KN}] First let $a=b$. The proof parallels the construction of convergents to a continued fraction; \cite[Chp. III]{Ch1978}.
By $G_n(a)$ of  \eqref{E:G3} and by formula \eqref{E:ghomogeneous},  for all $n\ge 3,$
\begin{equation}\label{E:Ghomogeneous}
\begin{array}{c}
G_{n}(a)= a(2-a) z h_{a} k_ag_{n-1}g_{n}
=c_{n,a}r^2z^{2n} \tau_{a}^{2n-4}/[w_n^{*}(a)w_{n-1}^{*}(a)], 
\end{array}
\end{equation}
for $c_{n,a}:=a^2(2-a)^2C_{n,a}C_{n-1,a}$,  with $C_{n,a}$ given by
\eqref{E:ghomogeneous}, where we have written $h_a k_a \omega_a^2= a(2-a)\tau_a$ by the definitions \eqref{E:khomega}. 
Further, for the full model we have:  
\begin{equation}\label{E:Hhomogeneous}
\begin{array}{c}
P(\mathbf{H}=n)= 4a\rho_{1,n}\gamma_n\rho_{n,0}, \ n\ge 2;  
\\  P(\mathbf{H}\ge n+1)=2a\rho_{1,n+1}, \ n\ge 2; 
\ \ P(H=1)= 2\frac{1}{2}(1-a)=1-a.
\end{array}
\end{equation}
By first principles we find $G_1(a)=r^2y^2z^2$ and $G_2(a)=r^2z^4k_a$. 
Therefore by  \eqref{E:KN} and \eqref{E:G},
$P(\mathbf{H}\le N)K_{N}(a)=\sum_{n=1}^{N}G_n(a)P(\mathbf{H}=n)$ is written by:
\begin{equation}\label{E:KNa}
\begin{array}{c}
 (1-a)r^2y^2z^2+\frac{a(1-a)}{2-a}r^2z^4k_a +\sum\limits_{n=3}^{N}c_{n,a}P(\mathbf{H}=n)\frac{r^2z^{2n} \tau_{a}^{2n-4}}{w_n^{*}(a)w_{n-1}^{*}(a)}.
\end{array}
\end{equation}
By  \eqref{E:Ghomogeneous}--\eqref{E:Hhomogeneous} and direct calculation, $c_{n,a}P(\mathbf{H}=n)=(1-a)a^{2n-2}$. Also, by \eqref{E:khomega} and \eqref{E:tauxbetahomogeneous}, $a^{2n-2}r^2z^{2n}\tau_a^{2n-4}=a^2 r^2z^4x_a^{n-2}$ while $k_a=\frac{a(2-a)}{w_{2}^{*}(a)}$, since $\omega_a=w_{2}^{*}(a)$. Therefore by \eqref{E:KNa}, $P(\mathbf{H}\le N)K_{N}(a)$ is written:
\begin{equation}\label{E:KNa1}
\begin{array}{l}
(1-a)r^2z^2\left(y^2+ \frac{a^2z^2}{w_{2}^{*}(a)}+\sum_{n=3}^{N}\frac{a^2 z^2 x_{a}^{n-2}}{w_n^{*}(a)w_{n-1}^{*}(a)}\right ).
\end{array}
\end{equation}
By \eqref{E:tauxbetahomogeneous}--\eqref{E:qwstar1} and direct calculation, we have that $y^2= q_{1}^{*}(a)/w_{1}^{*}(a)$ and $y^2+ a^2z^2/w_{2}^{*}(a)= q_{2}^{*}(a)/w_{2}^{*}(a)$. 
But, by Lemma \ref{L:bracketqw}, \emph{1}, we may write \\
$\begin{array}{c}
\frac{q_{n}^{*}}{w_n^{*}}-\frac{q_{n-1}^{*}}{w_{n-1}^{*}}=\frac{[w^{*}(a),q^{*}(a)]_{n-1}}{w_{n}^{*}w_{n-1}^{*}}=\frac{a^2z^2x_a^{n-2}}{w_n^{*}w_{n-1}^{*}},\ n\ge 1,
\end{array}$
where we suppressed the dependence on $a$ in $q_n^{*}$ and $w_n^{*}$. Therefore the sum in \eqref{E:KNa1} telescopes.  
Therefore, for all $N\ge 1$ the right side of \eqref{E:KNa1} becomes:
$(1-a)r^2z^2q_{N}^{*}(a)/w_{N}^{*}(a).$
Finally, apply \eqref{E:Hhomogeneous} and Proposition \ref{P:Pi} to compute $ P(\mathbf{H}\le N)=\frac{N(1-a)}{N-(N-1)a}$; so that,   by \eqref{E:KNa1}, Proposition \ref{P:KN} is proved. 
\par
We now indicate the additional steps required to  prove the Theorem \ref{T:KN}. First, with $N=f-1$ in Proposition \ref{P:KN}, and by Lemma \ref{L:bracketqw}, \eqref{E:KNa1} yields:  
\begin{equation}\label{E:TKN1}
\begin{array}{c}
\sum\limits_{n=1}^{f-1}G_nP(\mathbf{H}=n)=(1-a)r^2z^2\left(y^2+\sum\limits_{n=2}^{f-1}\frac{[w^{*}(a),q^{*}(a)]_{n-1}}{w_n^{*}(a)w_{n-1}^{*}(a)}\right ), 
\end{array}
\end{equation}
where here and in the rest of the proof we abbreviate $G_n=G_n(a,b)$.
Next by \eqref{E:G3} and  Proposition \ref{P:gProp}, \emph{III.1.1} and \emph{II.2},
\begin{equation}\label{E:TKN2}
\begin{array}{c}
G_f=a(2-a)z h_a k(a,b)g_{1,f}g_{f,0}=c_f\frac{r^2z^{2f}\left(a\tau_a\right)^{2f-4}}{w_{f-1}^{*}(a)\overline{w}_{f,0}}.
\end{array}
\end{equation}
Also by \eqref{E:G3} and  Proposition \ref{P:gProp}, \emph{I.1} and \emph{II.1}, for all $j\ge 1$, $G_{f+j}$ becomes:
\begin{equation}\label{E:TKN3}
\begin{array}{c}
a(2-a)zh_ak(b,b)g_{1,f+j}g_{f+j,0} =c_{f+j}\frac{r^2z^{2f+2j}\left(a\tau_a\right)^{2f-4}\tau(a,b)^2\left(b\tau_b\right)^{2j-2}}{\overline{w}_{1,f+j}\overline{w}_{f+j,0}}.
\end{array}
\end{equation}
In \eqref{E:TKN2} and \eqref{E:TKN3} the constants $c_f$ and $c_{f+1}$, respectively can be determined from Lemma \ref{L:C} since $G_n[\mathbf{1}]=1$. Indeed we find in this way, and by Definition \ref{D:Pi}, Proposition \ref{P:Pi}, \eqref{E:Hhomogeneous},  and direct calculation that $c_{f}P(\mathbf{H}=f)=a^2(1-b)$ and  $c_{f+j}P(H=f+j)=a^2b^2(1-b)$, $j\ge 1$. 
Thus by \eqref{E:tauxbetahomogeneous}, \eqref{E:tauxbeta}, and \eqref{E:TKN2}--\eqref{E:TKN3}, and since $w_{f-1}^{*}(a)= \overline{w}_{1,f}$, for all $j\ge 0$, $G_{f+j}P(\mathbf{H}=f+j)$ equals
\begin{equation}\label{E:TKN4}
\begin{array}{c}
a^2(1-b)\frac{r^2z^4x_a^{f-2}}{\overline{w}_{1,f}\overline{w}_{f,0}}, \ \mbox{ if} \  j=0;  \ \ 
a^2(1-b)\frac{r^2z^4x_a^{f-2}x(a,b) x_b^{j-1}}{\overline{w}_{1,f+j}\overline{w}_{f+j,0}},  \ \mbox{ if} \   j\ge 1.
\end{array}
\end{equation}
Therefore, by \eqref{E:TKN1}, \eqref{E:TKN4} and Lemma \ref{L:bracketqw}, for all $j\ge0$ there holds:
\begin{equation}\label{E:TKN5}
\begin{array}{c}
\sum\limits_{n=1}^{f+j}G_nP(\mathbf{H}=n)=(1-a)r^2z^2\left(y^2+\sum\limits_{n=2}^{f+j}\frac{[\overline{w},\overline{q}]_{n-1}}{\overline{w}_{1,n}\overline{w}_{n,0}}\right),
\end{array}
\end{equation}
where the fraction $\frac{1-b}{1-a}$ enters to form  $[\overline{w},\overline{q}]_{n-1}$ when $n=f+j$ for $j\ge0$ because we have factored out $(1-a)$ from the entire sum on the right.
But by the definition \eqref{E:bracketqw} and Lemma \ref{L:wbarID} we have that $\frac{[\overline{w},\overline{q}]_{n-1}}{\overline{w}_{1,n}\overline{w}_{n,0}}=\frac{\overline{q}_n}{\overline{w}_{1,n+1}}-\frac{\overline{q}_{n-1}}{\overline{w}_{1,n}}.$
Hence the sum in \eqref{E:TKN5} telescopes, and  thereby we finally obtain
$K_N(a,b)=P(\mathbf{H}\le N)^{-1} (1-a)\frac{r^2z^2\overline{q}_N}{\overline{w}_{1,N+1}}, \ N\ge f,$
where $P(\mathbf{H}\le N)^{-1} =\frac{(N+1-f)a+(f-1)b-(N-1)ab}{(N+1-f)a+(f-1)b-Nab}$ by \eqref{E:Hhomogeneous} and direct calculation. 
\qed
\end{proof}
\subsubsection{Proof of Corollary \ref{C:RUL3waysymm} }\label{S:Appl}
The unconditional joint generating function of the excursion statistics is $K:=E\{r^{\mathbf{R}}y^{\mathbf{V}}z^{\mathbf{L}}\}$. We  develop a simple representation of $K$ in the homogeneous case, as follows.
\begin{corollary}\label{C:K}
Let $a=b$ and define $\alpha_a:=\sqrt{\beta_a^2 -4x_a}$ for $x_a$ and $\beta_a$ given by \eqref{E:tauxbetahomogeneous}. Then
$\begin{array}{c}
K = \lim_{N\to\infty}K_N(a)=\left(1-\frac{1}{2}\beta_a-\frac{1}{2}\alpha_a\right)/(1-a).
\end{array}$
\end{corollary}
\begin{proof}[Corollary \ref{C:K}] Since we have explicitly seen in the proof of Proposition \ref{P:KN} that if $a=b$, then
$P(\mathbf{H}\le N)= \frac{N(1-a)}{N-(N-1)a}$, we have that the persistent random walk is recurrent: $\lim_{N\to \infty}P(\mathbf{H}\le N)=1$. So we obtain that 
$(*)\ \ K=\lim\limits_{N\to \infty}K_N(a)=(1-a)r^2z^2\lim_{N\to \infty}q_N^{*}/w_N^{*}.$
Here and in the rest of the proof we suppress dependence on $a$ when convenient; in particular denote $x=x_a$ and $\beta=\beta_a$. 
\par
We introduce a substitution variable $\theta$ as follows: 
\begin{equation}\label{E:complexexp}
\beta:=\sqrt{4x} \cos \theta; \ \ \beta\pm \alpha=\sqrt{4x}(\cos \theta \pm i \sin \theta) =\sqrt{4x}e^{\pm i\theta},
\end{equation}
with $\Im \theta <0$ for $|r|<1, |y|<1, z\neq 0$. The idea of the substitution \eqref{E:complexexp} may be found in \cite[p. 352]{Fe1968}. By $(\beta+\alpha)^n-(\beta+\alpha)^n=(4x)^{n/2}e^{i n \theta}(1+e^{-2i n \theta})$, the formulae \eqref{E:qwclosed} may be rewritten, where by our convention for the sign of $\Im \theta$, $1+e^{-2i n \theta} =1+o(1)$, as $n\to\infty$. We then substitute these expressions into the formulae
\eqref{E:qwstar1}  and find that $q_n^{*}(a)/w_n^{*}(a)$ is given by:
$$\frac{(y^2-q_0^{*})q_n(x,\beta)+q_0^{*}w_n(x,\beta)}{(1-w_0^{*})q_n(x,\beta)+w_0^{*}w_n(x,\beta)}=\frac{y^2(1-e^{-2i n \theta}) -\sqrt{x}q_0^{*}e^{-i\theta}(1-e^{-2i (n-1) \theta})}{1-e^{-2i n \theta} -\sqrt{x}w_0^{*}e^{-i\theta}(1-e^{-2i (n-1) \theta})}.$$
Therefore, using $e^{-i\theta}=(\beta-\alpha)/\sqrt{4x}$, we obtain $ \lim_{n\to\infty}\frac{q_n^{*}}{w_n^{*}}=\frac{y^2-q_0^{*}(\beta-\alpha)/2}{1-w_0^{*}(\beta-\alpha)/2}.$
Finally, we simplify this expression by multiplying both numerator and denominator by $1-w_0^{*}(\beta+\alpha)/2$. The new denominator becomes $1+x_aw_0^{*}(a)^2-\beta_a w_0^{*}(a)=(1-a)^2r^2z^2/\tau_a^2$, by direct calculation. Therefore by bringing the $\tau_a^2$ of this last expression to the numerator we obtain that $\lim_{n\to\infty}(1-a)^2r^2z^2q_n^{*}/w_n^{*}= \tau_a^{2}(y^2-q_0^{*}(\beta-\alpha)/2)(1-w_0^{*}(\beta+\alpha)/2)$, or 
$$\begin{array}{c}
[y^2-\beta_a(q_0^{*}+y^2w_0^{*})/2+x_a w_0^{*}q_0^{*}] \tau_a^{2} + \alpha_a[q_0^{*}-y^2w_0^{*}]\tau_a^{2}/2  =I +\alpha_a II,
\end{array}$$
after cancellation of terms $\pm q_0^{*}w_0^{*}\alpha\beta/4$. By direct calculation we find that $I=1-\frac{1}{2}\beta_a$, and $II=-\frac{1}{2}$. 
Hence by $(*)$ the proof is complete. \qed
\end{proof}
We may view the excursion statistics in the case $a=b$ by the way they are weighted relative to one another. Indeed, a specific excursion path of $2n$ steps and $2k$ runs is weighted with the  probability $\frac{1}{2}a^{2n-2k}(1-a)^{2k-1}$, for $k$ peaks and $k-1$ valleys. In the unweighted case $a=\frac{1}{2}$, it is known that the joint distribution of $(\mathbf{L},\mathbf{R})$ is essentially the same as that of $(\mathbf{L},\mathbf{L}-\mathbf{R})$ [see \cite{OEIS}, A001263; symmetry of the Narayana numbers].
\begin{proof}[Corollary \ref{C:RUL3waysymm}] We establish the joint generating function identity in the unweighted case via a direct calculation. Let $r$, $u$ and $z$ belong to the unit circle. By applying Corollary \ref{C:K} with $a=\frac{1}{2}$ we obtain the joint generating function of runs, long runs, and steps by
$$\begin{array}{l}
K(\frac{1}{2})[ru,1/u,z]= \frac{1}{16}\left(16 - 4 z^2 + 4 r^2 z^2 + r^2 z^4 - 2 r^2 u z^4 + r^2 u^2 z^4 - S\right ),
\end{array}$$
with $S$ given by $S=S_1S_2S_3S_4$,  for: \\
$$\begin{array}{c}
S_1=\sqrt{4 + 2 z + 2 r z +r z^2 - r u z^2},\ S_2=\sqrt{ 4 + 
        2 z - 2 r z - r z^2 + r u z^2},\\  S_3=\sqrt{4 - 2 z + 2 r z - r z^2 + 
        r u z^2}, \ S_4=\sqrt{4 - 2 z - 2 r z + r z^2 -r u z^2}.
\end{array}$$
On the other hand, with the very same main term $S$, we have
$$\begin{array}{l}
K(\frac{1}{2})[u/r,1/u,rz]= \frac{1}{16}\left(16 + 4 z^2 - 4 r^2 z^2 + r^2 z^4 - 2 r^2 u z^4 + r^2 u^2 z^4 - S\right ).
\end{array}$$
The two generating functions differ by $K(\frac{1}{2})[ru,1/u,z]-K(\frac{1}{2})[u/r,1/u,rz]=\frac{1}{2}z^2(r^2-1)$. The difference is mirrored only in the event that $\mathbf{L}=2$, when it happens that  $\mathbf{R}=2$ and $\mathbf{U}=0$. Thus \eqref{E:RUL3waysymm} holds for  $a=\frac{1}{2}$ and $n\ge 2$.
\par
Perhaps the simplest way to obtain  \eqref{E:RUL3waysymm} for $a\neq \frac{1}{2}$ is to apply \eqref{E:RUL3waysymm} for the case $a=\frac{1}{2}$. 
Consider an excursion path $\Gamma$ with $\mathbf{L}(\Gamma)=2n$ and $\mathbf{L}(\Gamma)-\mathbf{R}(\Gamma)=2k$. Then $P_a(\Gamma)= \frac{1}{2}a^{2k}(1-a)^{2n-2k-1}$.  Here $\mathbf{R}(\Gamma)-1=2n-2k-1$ counts the number of turns in the path, so is the exponent of $(1-a)$ under $P_a$.  Alternatively, $\mathbf{L}(\Gamma)-\mathbf{R}(\Gamma)$ is the total length of long runs minus the number of long runs in $\Gamma$, and this gives the exponent of $a$ in $P_a( \Gamma)$. 
If $2n\ge 4$, then by the first part of the proof there are exactly as many paths $\Gamma$ with the joint information $\mathbf{L}(\Gamma)=2n$, 
$\mathbf{L}(\Gamma)-\mathbf{R}(\Gamma)=2k$, and $\mathbf{U}(\Gamma) =\ell$, as there are paths $\Gamma '$ with $\mathbf{L}(\Gamma ')=2n$, $\mathbf{R}(\Gamma ')=2k$, and $\mathbf{U}(\Gamma ') =\ell$. Therefore, since for any such path $\Gamma '$, the probability assigned by the probability measure $P_{1-a}$ yields
 $P_{1-a}(\Gamma ')=\frac{1}{2}a^{2k-1}(1-a)^{2n-2k}$, we have that $aP_{1-a}(\Gamma ')=(1-a)P_a(\Gamma)$, $\forall\ \Gamma$ with $\mathbf{L}(\Gamma)\ge 4$.  Hence \eqref{E:RUL3waysymm} holds.  
 \qed
 \end{proof}
\section{Proofs of Theorems \ref{T:T2} and \ref{T:T3}.}\label{S:T2}
\begin{proof}[Theorem \ref{T:T2}.]  We fix $t\in \mathbb{R}$. All big oh terms in the proof will refer to the parameter $N\to \infty$ with implied constants depending only on $a$, $b$, and $t$.  Since, by \cite{Moh1955}, for fixed $m>0$ and $f\sim \eta N\to\infty$, $P(\mathbf{X}_j=0\mbox{ before }\mathbf{X}_j=f|\mathbf{X}_0=m)\to 1$,  we may assume that $\mathbf{X}_0=0$. Let 
\begin{equation}\label{E:charfcnX}
\begin{array}{c}
r_N:=e^{-it(2-a-b)/((1-a)(1-b)N)},  \ y_N:=e^{it/((1-a)(1-b)N)}, \ z_N:=e^{it/N}.
\end{array}
\end{equation}
Since $\{(1+\frac{1}{N})X_{N+1} \}$ converges in distribution if and only if $\{X_N\}$ does,  by \eqref{E:Xab} it suffices to establish that $E\{e^{it(1+\frac{1}{N})X_{N+1}}\}=\hat{\varphi}(t)$, as $N\to \infty$. It is clear that $a(2-a)z_Nh_a[r_N,y_N,z_N]\to 1$ as $N\to \infty$. Therefore, by \eqref{E:meandergf}, we must show that  $\lim_{N\to\infty}g_{0,N}[r_N,y_N,z_N]$ equals the limit in \eqref{E:T2}. 
By Proposition \ref{P:gProp} \emph{I.1}, we have a formula for $g_{0,N}$, and by Proposition \ref{P:wformula} we have a formula for its denominator $\overline{w}_{0,N}$. 
The main work is in calculating an asymptotic expression for $\overline{w}_{0,N}[r_N,y_N,z_N]$.
\par
We now make substitutions analogous to \eqref{E:complexexp}, one for each stratum:
\begin{equation}\label{E:thetaab}
\begin{array}{c}
\cos(\theta_1):=\beta_a/\sqrt{4x_a}\ ;\ \  \cos(\theta_2):=\beta_b/\sqrt{4x_b}; \\
\beta_a\pm \alpha_a=\sqrt{4x_a}e^{\pm i\theta_1}\ ; \ \ \beta_b\pm \alpha_b=\sqrt{4x_b}e^{\pm i\theta_2},
\end{array}
\end{equation}
where all functions on the right sides of these expressions are composed with $[r_N,y_N,z_N]$ of \eqref{E:charfcnX}. Here we write $\sqrt{4x_a}$ as a shorthand for the expression $2az\tau_a$; see \eqref{E:tauxbetahomogeneous}. 
Note that the coefficients in \eqref{E:Xab} have been chosen such that the first order term of the Taylor expansions about $t=0$ of the substitutions $\cos \theta_j [r_N,y_N,z_N]$ , $j=1,2$, do in fact  vanish  in the following:
\begin{equation}\label{E:asympcostheta}
\begin{array}{c}
\cos\theta_1= 1+\frac{1}{2}\frac{\sigma_1^{2} t^2}{(1-b)^2N^2}+O\left(\frac{1}{N^3}\right); \  \ \cos\theta_2 =1+\frac{1}{2}\frac{\sigma_2^{2} t^2}{(1-a)^2N^2}+O\left(\frac{1}{N^3}\right),
\end{array}
\end{equation}
where $\sigma_1^2$ and $\sigma_2^2$ are as defined in the statement of the theorem, and we obtain  \eqref{E:asympcostheta} by direct computation. 
Therefore by \eqref{E:asympcostheta}, and by applying the Taylor expansion of $\arccos(u)$ about $u=1$, we find that $\theta_1$ and $\theta_2$ are both of order $1/N$ as follows:
\begin{equation}\label{E:asymptheta}
\begin{array}{c}
\theta_1 = i \frac{\sigma_1 t}{(1-b)N} +O\left(\frac{1}{N^3}\right); \ \theta_2 = i\frac{\sigma_2 t}{(1-a)N} +O\left(\frac{1}{N^3}\right).
\end{array}
\end{equation}
By Proposition \ref{P:wformula},  
\begin{equation}\label{E:star} \overline{w}_{0.N}=d_1(f)q_{N-f}^{*}(b)+d_2(f)w_{N-f}^{*} (b).
\end{equation}
We focus first on the coefficients  $d_{j}(f)$, which are written in terms of $w^{*}_f(a)$ and $w^{*}_{f+1}(a)$ by \eqref{E:M}. 
By \eqref{E:qwclosed} and \eqref{E:qwstar1}, suppressing dependence on $a$,  $w_f^{*}=(1-w_0^{*})q_f+w_0^{*}(q_f-xq_{f-1})=q_f(x,\beta)-w_0^{*}xq_{f-1}(x,\beta)$. Thus by \eqref{E:qwclosed},  and \eqref{E:thetaab}, 
\begin{equation}\label{E:wstarf}
\begin{array}{c}
w_f^{*}(a)=2i\alpha_a^{-1}\left(az\tau_a\right)^f\left\{\sin f\theta_1 -\sqrt{x_a}w_0^{*}(a)\sin (f-1)\theta_1 \right \}\\ 
w_{f+1}^{*}(a)=2i\alpha_a^{-1}\left(az\tau_a\right)^f\sqrt{x_a}\left\{\sin (f+1)\theta_1 -\sqrt{x_a}w_0^{*}(a)\sin f\theta_1 \right \};
\end{array}
\end{equation}
with verification by direct algebra for $q_f(x_a,\beta_a) = 2i \alpha_{a}^{-1}\left(az\tau_a\right)^f \sin(f\theta_1)$, and with $\sqrt{x_a}$ to stand for a factor of $\left(az\tau_a\right)$. 
Next denote 
\begin{equation}\label{E:Lambda1}
\begin{array}{c}
 e_j=e_j(f):=\frac{d_j(f)}{\Lambda_1},\ j=1,2, \ \mbox{for } \Lambda_1:=2i\alpha_{a}^{-1}\left(az\tau_a\right)^{f}=(\sin \theta_1)^{-1}\left(az\tau_a\right)^{f-1},
 \end{array}
\end{equation}
since $\alpha_a= i\sqrt{4x_a}\sin \theta_1=2i az \tau_a\sin\theta_1$. By \eqref{E:wstarf}--\eqref{E:Lambda1} and direct algebra, through \eqref{E:M} we can write an expression for $e_{j}$
as follows:
\begin{equation}\label{E:d1d2}
\begin{array}{l}
\left(\mu_{j,1}-\mu_{j,2}x_aw_0^{*}(a)\right)\sin f\theta_1 +\sqrt{x_a}\left[ \mu_{j,2}\sin (f+1)\theta_1-\mu_{j,1}w_0^{*}(a) \sin (f-1)\theta_1\right].
\end{array}
\end{equation}
Next we apply the trigonometric identity for the sine of a sum or difference to  $\sin (f+1)\theta_1$ and $\sin (f-1)\theta_1$ in \eqref{E:d1d2}. At this point we also introduce some abbreviations to keep the notation a bit compact. Thus write  
\begin{equation}\label{E:s1c1} 
\begin{array}{c}
\mathbf{s}_1:=\sin f\theta_1 ; \ \ \mathbf{c}_1:=\cos f\theta_1.
\end{array}
\end{equation}
We rewrite 
\eqref{E:d1d2}, with abbreviation $w_0^{*}=w_0^{*}(a)$, by collecting terms with a factor $\sqrt{x_a}$. Thus for each $j=1,2$, 
\begin{equation}\label{E:d1d2reduced} 
\begin{array}{l}
e_j=(\mu_{j,1}-\mu_{j,2}x_aw_0^{*})\mathbf{s}_1  \\  \ \ \ \  +\sqrt{x_a}\left\{ \mu_{j,2}(\mathbf{s}_1\cos\theta_1+\mathbf{c}_1\sin \theta_1)-\mu_{j,1}w_0^{*} (\mathbf{s}_1\cos\theta_1-\mathbf{c}_1\sin\theta_1)\right\}.
\end{array}
\end{equation}
We introduce a book--keeping notation for the coefficient $t_j$ of the variable $\mathbf{x}_j$ in square brackets, within a linear expression $\sum_{i}t_i \mathbf{x}_i $ in parentheses: $ [ \mathbf{x}_j ]\left( \sum_{i}t_i \mathbf{x}_i \right) =t_j.$ Our method for $e_j$ is to asymptotically expand 
$[\mathbf{s}_1](e_j)$ and $[\mathbf{c}_1 \sin \theta_1](e_j) $ by \eqref{E:d1d2reduced}. We will treat $\sin \theta_1$ separately from the asymptotic expansions of the other terms due to the convenient fact that, by \eqref{E:asymptheta}, we have $\sin\theta_1 =\theta_1+O(N^{-3})$, and this will suffice for our purposes. Note that by  \eqref{E:asymptheta} and \eqref{E:s1c1}, and $f\sim \eta N$, $\mathbf{s}_1$ and $\mathbf{c}_1$ are both $O(1)$. Further, by direct calculation, $\mu_{i,j}$ are polynomial, and $q_0^{*}(a)$ and $w_0^{*}(a)$ only involve negative powers of $\tau_a$, where $\tau_a[\mathbf{1}]=1$. Thus the Taylor expansions of $[\mathbf{s}_1](e_j)$ and  $[\mathbf{c}_1\sin\theta_1](e_j)$ about $t=0$ are well behaved.
\par
We next find reduced expressions for the terms $q_{N-f}^{*}(b)$ and $w_{N-f}^{*}(b)$ of \eqref{E:star}.  The approach is as above, but now with $b$ in place of $a$, $N-f$ in place of $f$, and using the second substitution $\theta_2$ in \eqref{E:thetaab}.  Similar to \eqref{E:Lambda1} we introduce 
\begin{equation}\label{E:Lambda2} 
\begin{array}{c}
q^{*}:=\frac{q_{N-f}^{*}(b)}{\Lambda_2}, \ w^{*}:=\frac{w_{N-f}^{*}(b)}{\Lambda_2}; \   \Lambda_2:=2i\alpha_{b}^{-1}\left(bz\tau_b\right)^{N-f}=(\sin \theta_2)^{-1}\left(bz\tau_b\right)^{N-f-1}.
\end{array}
\end{equation}
Similar as for \eqref{E:wstarf}, by \eqref{E:qwclosed} and both lines of \eqref{E:qwstar1} applied in turn, and  \eqref{E:Lambda2}, 
\begin{equation}\label{E:qwstarNminusf}
\begin{array}{c}
q^{*}:=y^2\sin (N-f)\theta_2 -\sqrt{x_b}q_0^{*}(b)\sin (N-f-1)\theta_2 ,\\ 
w^{*}:=\sin (N-f)\theta_2 -\sqrt{x_b}w_0^{*}(b)\sin (N-f-1)\theta_2 .
\end{array}
\end{equation}
Introduce abbreviations also for the second stratum sines and cosines:
\begin{equation}\label{E:s2c2} 
\begin{array}{c}
\mathbf{s}_2:=\sin (N-f)\theta_2 ; \ \ \mathbf{c}_2:=\cos(N-f)\theta_2.
\end{array}
\end{equation}
We illustrate the book--keeping method by expanding  $\sin (N-f-1)\theta_2=\mathbf{s}_2\cos\theta_2-\mathbf{c}_2\sin\theta_2$ to obtain by \eqref{E:qwstarNminusf},
$$\begin{array}{c}
[\mathbf{s}_2]\left(q^{*}\right) =y^2-\sqrt{x_b}q_0^{*}(b)\cos\theta_2; \ \ [ \mathbf{c}_2 \sin\theta_2 ]\left(q^{*}\right)=\sqrt{x_b}q_0^{*}(b); \end{array}$$
$$\begin{array}{c} [\mathbf{s}_2]\left(w^{*}\right)=1-\sqrt{x_b}w_0^{*}(b)\cos\theta_2; \ \ 
 [\mathbf{c}_2\sin\theta_2]\left(w^{*}\right)=\sqrt{x_b}w_0^{*}(b). 
\end{array}$$
To handle the asymptotic expansions for the four terms on the right side of \eqref{E:star}, we expand the coefficients of $\mathbf{s}_1$, $\mathbf{c}_1\sin \theta_1$, $\mathbf{s}_2$, and $\mathbf{c}_2\sin \theta_2$ by direct computation and thereby find  
\begin{equation}\label{E:w0Nformulafinal}
\begin{array}{c}
\frac{\overline{w}_{0,N}}{\Lambda_1\Lambda_2}=O(N^{-2})+\left [\left(-(1-a)(1-b)+2(1-ab)\frac{it}{N}\right )\mathbf{s}_1\right ] \left[\left (1+\frac{2}{1-a}\frac{it}{N} \right ) \mathbf{s}_2\right ]+ \\
\left[\left (1-a-\frac{a(b-a)}{1-b}\frac{it}{N}\right )\mathbf{s}_1+ a \mathbf{c}_1\sin\theta_1\right]\left[\left(1-b-\frac{b(2-a-b)}{1-a}\frac{it}{N}\right)\mathbf{s}_2+b\mathbf{c}_2\sin\theta_2\right].
\end{array}
\end{equation}
Since, by \eqref{E:asymptheta}, $\sin\theta_1$ and $\sin\theta_2$ are of order $1/N$, observe that the two terms of order 1 on the right hand side of  \eqref{E:w0Nformulafinal}  are of form $\pm(1-a)(1-b)$ and therefore cancel. Also, since $\sin\theta_j=\theta_j+O(N^{-3})$ for $\theta_j$ are given by \eqref{E:asymptheta}, we substitute these  relations into \eqref{E:w0Nformulafinal} and collect the order $1/N$ terms to find by direct asymptotics that:
\begin{equation}\label{E:w0Nformulafinal1}
\frac{\overline{w}_{0,N}}{\Lambda_1\Lambda_2}=\left \{a\sigma_1\mathbf{c}_1\mathbf{s}_2+b\sigma_2\mathbf{s}_1\mathbf{c}_2+(b-a)^2 \mathbf{s}_1\mathbf{s}_2\right \}\frac{it}{N}+O(N^{-2}).
\end{equation}
To render a partial check on the book--keeping procedure for \eqref{E:w0Nformulafinal1},  
write out a formula for $e_j=e_j(a, b, r_N, y_N, z_N, \mathbf{s}_1, 
   \mathbf{c}_1 \sin \theta_1)$ of \eqref{E:d1d2reduced} by leaving $\sin \theta_1$ as an auxiliary variable. Then $\sin\theta_j$ is replaced by the order $1/N$ term of \eqref{E:asymptheta}, whereas $\cos \theta_j$ is defined exactly by \eqref{E:thetaab}. So, expand $e_1q^{*}+e_2w^{*}$ as 
    $$\begin{array}{c}e_1\left(
     [\mathbf{s}_2](q^{*})\mathbf{s}_2+ 
    [\mathbf{c}_2 \sin\theta_2](q^{*}) \mathbf{c}_2 \sin\theta_2 \right)+e_2\left(
     [\mathbf{s}_2](w^{*})\mathbf{s}_2+ 
    [\mathbf{c}_2 \sin\theta_2](w^{*}) \mathbf{c}_2 \sin\theta_2 \right),\end{array}$$ 
and apply a Taylor series about  $t=0$  to recover \eqref{E:w0Nformulafinal1}; see \cite{Mor2018}.
 \par
Now plug \eqref{E:w0Nformulafinal1} into the formula for $g_{0,N}$ in Proposition \ref{P:gProp}.\emph{I.1}, apply 
Proposition \ref{P:Pi}.\emph{I.1} to rewrite $\Pi_{0,N}$, and recall  $\Lambda_j$ in \eqref{E:Lambda1} and \eqref{E:Lambda2}. So
$$g_{0,N}=\frac{\omega_a \tau(a,b) r z^2}{a(2-a)\tau_a}\left(\frac{ \sin \theta_1\sin \theta_2[(N-f)a+fb-(N-1)ab]}{\left [ a\sigma_1\mathbf{c}_1\mathbf{s}_2+b\sigma_2\mathbf{s}_1\mathbf{c}_2+(b-a)^2 \mathbf{s}_1\mathbf{s}_2\right ] \frac{it}{N}+O(N^{-2})}\right ).$$
Finally, to find the limit as $N\to \infty$ of this last expression,  we substitute \eqref{E:asymptheta} into the definitions \eqref{E:s1c1} and \eqref{E:s2c2}, and again employ $\sin \theta_j \sim \theta_j $. We note: $\lim_{N\to\infty}\omega_a [a(2-a)]^{-1}r_N z_N^2\tau(a,b)\tau_a^{-1}=1$, since $\omega_a[\mathbf{1}]=a(2-a)$ and $\tau(a,b)[\mathbf{1}]=1$. Since by assumption $f\sim\eta N$, we have $[(N-f)a+fb-(N-1)ab]\sim N[(1-\eta)a+\eta b-ab]$, and since by \eqref{E:asymptheta}, $\theta_1\theta_2\sim i^2 \frac{\sigma_1\sigma_2}{(1-a)(1-b)} t^2N^{-2} $, we obtain, as $N\to \infty,$
$$
g_{0,N}\sim \frac{i^2t^2}{N} \frac{\sigma_1\sigma_2}{(1-a)(1-b)}\frac{(1-\eta)a+\eta b-ab }{
 \left [ a\sigma_1\mathbf{c}_1\mathbf{s}_2+b\sigma_2\mathbf{s}_1\mathbf{c}_2+(b-a)^2 \mathbf{s}_1\mathbf{s}_2\right ] \frac{it}{N}}.
$$
Here we use implicitly that  $\sin (i x)=i\sinh (x)$ and $\cos(i x)=\cosh(x)$, so that by \eqref{E:asymptheta}, \eqref{E:s1c1} and \eqref{E:s2c2}, and by definition of $\kappa_1$ and $\kappa_2$, $\mathbf{s}_j\sim i \sinh(\kappa_j t)$, $j=1,2$,   and $\mathbf{c}_j\sim \cosh(\kappa_j t)$, $j=1,2$. Thus  
we obtain, $\lim\limits_{N\to\infty}g_{0,N}[r_,s_N,t_N]= \hat{\varphi}(t) $, 
for $\hat{\varphi}(t)$ given by \eqref{E:T2}.  
\qed 
\end{proof}
\begin{proof}[Corollary \ref{C:homogeneouslimit}] We now assume that $a=b$ and consider the random variable $sY_{1,N}+tY_{2,N}$ defined by \eqref{E:Delta1Delta2} in place of $tX_N$ in the proof of Theorem \ref{T:T2}. By the definition \eqref{E:Delta1Delta2}  we write
$sY_{1,N}+tY_{2,N}=\frac{1}{N}\left ({t\cal L}'_N+\frac{(1-a)s-(2-a)t}{(1-a)}{\cal R}'_N+\frac{t-s}{(1-a)}{\cal V}'_N\right ).$
Accordingly, define
\begin{equation}\label{E:rhN}
r_{s,t,N}:=e^{i((1-a)s-(2-a)t)/((1-a)N)},  \ y_{s,t,N}:=e^{i(t-s)/((1-a)N)}, \ z_{s,t,N}:=e^{it/N}.
\end{equation}
It suffices to prove that, for each fixed pair of real numbers $s,t\in \mathbb{R}$,  $\lim_{N\to\infty} g_{0,N}(a,a)[r_{s,t,N},y_{s,t,N},z_{s,t,N}]$ exists and is given by the right side of \eqref{E:Jointchfcn}. 
Define $\theta=\theta_{s,t,N} $ via $\cos\theta=\beta_a/\sqrt{4x_a}$, where the functions $\beta_a$ and $x_a$  are composed with the complex exponential terms in \eqref{E:rhN}. It follows by making a direct calculation that 
\begin{equation}\label{E:costhetahN}
\begin{array}{c}
 \cos \theta=1+\frac{1}{2}\frac{(1-a)s^2+at^2}{N^2}+O\left(\frac{1}{N^3}\right); \  \theta = i \frac{\sqrt{(1-a)s^2+at^2}}{N} +O\left(\frac{1}{N^3}\right).
 \end{array}
\end{equation}
Since the model is homogeneous, we need only apply the first line of \eqref{E:wstarf} with $f:=N$ to obtain
\begin{equation}\label{E:wstarhN}
w_N^{*}(a)=(\sqrt{x_a}\sin\theta)^{-1}[az\tau_a]^N\left\{\sin N\theta-\sqrt{x_a}w_0^{*}(a)\sin (N-1)\theta \right \}.
\end{equation}
Expand $\sin (N-1)\theta=\mathbf{s}\cos \theta-\mathbf{c}\sin \theta$, for $\mathbf{s}:=\sin N\theta$ and $\mathbf{c}:=\cos N\theta$. Put $\Lambda:=(\sin\theta)^{-1}\left(az\tau_a\right)^{N-1}$. After direct calculation we find $1-\sqrt{x_a}w_0^{*}(a)=1-a +O(N^{-1})$. Therefore by \eqref{E:wstarhN} we have 
$$\begin{array}{c}
\frac{w_N^{*}(a)}{\Lambda}=\mathbf{s} -\sqrt{x_a}w_0^{*}(a)(\mathbf{s}\cos\theta-\mathbf{c} \sin\theta)=(1-a)\mathbf{s}+O\left(\frac{1}{N}\right) 
\end{array}.$$
Note that there is no cancellation of the order 1 term in this expression.
Now plug $\frac{w_N^{*}(a)}{\Lambda}$ into \eqref{E:ghomogeneous} to obtain 
$$g_{0,N}=\frac{\omega_a}{a(2-a)}r z\tau_a^{-1}\frac{(N-(N-1)a) \sin\theta} {(1-a)\mathbf{s} + O(N^{-1})}.$$
Finally apply the asymptotic expression for $\theta$ in \eqref{E:costhetahN} and let $N\to \infty$.
\qed  \end{proof}
\begin{proof}[Theorem \ref{T:T3}]
 By the same reasoning given at the outset of the proof of Theorem \ref{T:T2}, we may assume that $\mathbf{X}_0=0$. 
By the fact that the absolute value process starts afresh at the end of each excursion, we have that $1+{\cal M}_N$ is a standard geometric random variable with success probability $P(\mathbf{H}\ge N)$. Thus
\begin{equation}\label{E:PM}
\begin{array}{c}
P({\cal M}_N=\nu)=\left [P(\mathbf{H}< N) \right ]^{\nu}P(\mathbf{H}\ge N), \mbox{ \ } \nu=0,1,2,\dots.
\end{array}
\end{equation}
Let $\mathbf{L}_N$, $\mathbf{R}_N$, and $\mathbf{V}_N$, respectively,  be random variables for the number of steps, runs, and short runs, in an excursion, given that the height of the excursion is at most $N-1$.  
Therefore, in distribution, we may write: 
$$\begin{array}{c}
{\cal R}_N=\sum_{\nu=0}^{{\cal M}_N}\mathbf{R}^{(\nu)}, \ \   {\cal V}_N=\sum_{\nu=0}^{{\cal M}_N}\mathbf{V}^{(\nu)},  \ \ 
{\cal L}_N=\sum_{\nu=0}^{{\cal M}_N}\mathbf{L}^{(\nu)},
\end{array}$$
where  $\mathbf{R}^{(1)},\mathbf{R}^{(2)},\dots$; $\mathbf{V}^{(1)},\mathbf{V}^{(2)},\dots$; and $\mathbf{L}^{(1)},\mathbf{L}^{(2)},\dots$, respectively, are sequences of independent copies of $\mathbf{R}_{N}$, $\mathbf{V}_{N}$, and $\mathbf{L}_{N}$. Since the random variables $\mathbf{R}_{N}$, $\mathbf{V}_{N}$, and $\mathbf{L}_{N}$ already have built into their definitions the condition $\{\mathbf{H}\le N-1\}$, the probability generating function $K_{N-1}=E\{r^{\mathbf{R}_{N}}y^{\mathbf{V}_{N}}z^{\mathbf{L}_{N}}\}$ is calculated by Theorem \ref{T:KN}. Thus by \eqref{E:PM}, and by calculating a geometric sum there holds:
\begin{equation}\label{E:LastVisitGenFcn}
\begin{array}{c}
E\{r^{{\cal R}_N}y^{{\cal V}_N}z^{{\cal L}_N}u^{{\cal M}_N}\} = \sum\limits_{\nu=0}^{\infty}P({\cal M}_N=\nu)\left(uK_{N-1}\right )^{\nu} = \frac{P(\mathbf{H}\ge N)}{1-uP(\mathbf{H}< N)K_{N-1}[r,y,z]}.
\end{array}
\end{equation}
We define $(r_N,y_N,z_N)$ by \eqref{E:charfcnX}, and also set $u_N:=e^{-it a(b-a)/[(1-a)(1-b)N]}$.  By \eqref{E:lastvisitXab}, it suffices to show that $\lim_{N\to\infty}E\{e^{it(1+1/N){\cal X}_{N+1}}\}=\hat{\psi}(t)/\hat{\varphi}(t)$; see \eqref{E:jointcharacfcncalX}. We define $\theta_1$ and $\theta_2$ by \eqref{E:thetaab},  so that also \eqref{E:asympcostheta}--\eqref{E:asymptheta}  hold.  By the statement of Theorem \ref{T:KN} we must replace the calculation of $\overline{w}_{0,N}$, starting with \eqref{E:star}, with instead  $\overline{w}_{1,N+1}$. 
However, by \eqref{E:M}, \eqref{E:wn0formula}, and \eqref{E:star}, the difference in the two calculations is simply accounted for by replacing $f$ by $f-1$ in the calculation of  $\overline{w}_{0,N}$, because $j$ in  \eqref{E:wn0formula} for $\overline{w}_{1,N+1}$ is determined by $j=N+1-f=N-(f-1)$, so $\frac{1}{\Lambda_2}\overline{w}_{1,N+1}=d_1(f')q^{*} +d_2(f')w^{*}$ with $f':=f-1$ in place of $f$ in both \eqref{E:M} and \eqref{E:Lambda2}. This is reflected by the fact that, by Lemma \ref{L:wbarID}, $\overline{w}_{1,N+1}=\overline{w}_{N,0}$. 
We must now also calculate $\overline{q}_N=d_{q,1}(f)q_{N-f+1}^{*}(b)+d_{q,2}(f)w_{N-f+1}^{*}(b)$ given by Lemma \ref{L:Qformula}, with 
$d_{q,j}(f)=\mu_{j,1}q_{f-1}^{*}(a)+\mu_{j,2}q_{f}^{*}(a)$, $j=1,2$,  
defined by  \eqref{E:Qprime} in the proof of Lemma \ref{L:Qformula}. In summary, $f'=f-1$ yields
$(\dagger)\ \ \overline{q}_{N}=d_{q,1}(f'+1)q_{N-f'}^{*}(b)+d_{q,2}(f'+1)w_{N-f'}^{*}(b)$.
Thus, because we simply replace $f$ by $f-1$ in the required substitutions, and since $f\sim \eta N$, we will not change the name of $f$. With this understanding, we may use the calculation of $\overline{w}_{0,N}$ in \eqref{E:star}--\eqref{E:w0Nformulafinal1} verbatim in place of the calculation of $\overline{w}_{1,N+1}$, and we will do this without changing the names of $e_j$, $q^{*}$, $w^{*}$ and $\Lambda_j$; see \eqref{E:Lambda1} and \eqref{E:Lambda2}. Further with this understanding, by $(\dagger)$, with $f$ now recouping the role of $f'$, and with $q^{*}$ and $w^{*}$ defined by \eqref{E:Lambda2}, we have $\frac{1}{\Lambda_2}\overline{q}_{N}=d_{q,1}q^{*}+d_{q,2}w^{*}$ for
\begin{equation}\label{E:dq}
d_{q,j}:=\mu_{j,1}q_{f}^{*}(a)+\mu_{j,2}q_{f+1}^{*}(a). 
\end{equation}
Here by \eqref{E:qwclosed}, \eqref{E:qwstar1} and \eqref{E:thetaab}, in analogy with \eqref{E:wstarf}, we have
$$\begin{array}{c}
q_{f}^{*}(a)=2i\alpha_a^{-1}\left(az\tau_a\right)^{f}\left\{y^2\sin f\theta_1 -\sqrt{x_a}q_0^{*}(a)\sin (f-1)\theta_1 \right \}\\ 
q_{f+1}^{*}(a)=2i\alpha_a^{-1}\left(az\tau_a\right)^{f}\sqrt{x_a}\left\{y^2\sin (f+1)\theta_1 -\sqrt{x_a}q_0^{*}(a)\sin f \theta_1 \right \}.
\end{array}$$
Denote $e_{q,j}:=d_{q,j}/\Lambda_1$. Therefore, by \eqref{E:dq}, the definition of $\Lambda_1$ in \eqref{E:Lambda1}, and these equations for $q_{f}^{*}(a)$ and $q_{f+1}^{*}(a)$, 
\begin{equation}\label{E:dq1dq2}
\begin{array}{l}
e_{q,j}=\left(y^2\mu_{j,1}-\mu_{j,2}x q_0^{*}\right)\sin f\theta_1 +\sqrt{x}\left\{ y^2\mu_{j,2}\sin (f+1)\theta_1-\mu_{j,1}q_0^{*} \sin (f-1)\theta_1\right\},
\end{array}
\end{equation}
where $x=x_a$ and $q_0^{*}=q_0^{*}(a)$. Rewrite  \eqref{E:dq1dq2} by applying the notations \eqref{E:s1c1}. Thus $e_{q,j}$ is written, with dependence on $a$ suppressed, by
\begin{equation}\label{E:dq1dq2reduced} 
\begin{array}{l}
 (y^2\mu_{j,1} - \mu_{j,2}xq_0^{*})\mathbf{s}_1  +\sqrt{x}\{ y^2\mu_{j,2}(\mathbf{s}_1\cos\theta_1+\mathbf{c}_1\sin \theta_1)-\mu_{j,1}q_0^{*} (\mathbf{s}_1\cos\theta_1-\mathbf{c}_1\sin\theta_1)\}
\end{array}
\end{equation}
In summary, by \eqref{E:dq}, we have $\overline{q}_{N}/(\Lambda_1\Lambda_2)=e_{q,1}q^{*}+e_{q,2}w^{*},$ for $e_{q,j}$ in \eqref{E:dq1dq2reduced}, and $\Lambda_j$ defined by  \eqref{E:Lambda1} and \eqref{E:Lambda2}.
\par
To guide the asymptotic expansions of \eqref{E:dq1dq2reduced}  we rewrite \eqref{E:LastVisitGenFcn}
by substituting the last line of the proof of Theorem \ref{T:KN}: 
 \begin{equation}\label{E:jointcharacfcncalX}
E\{e^{it(1+1/N){\cal X}_{N+1}}\}
=\frac{P(\mathbf{H}\ge N+1)\overline{w}_{1,N+1}}{\overline{w}_{1,N+1}-(1-a)u_Nr_N^{2} z_N^{2} \overline{q}_{N}}.
\end{equation}
It turns out 
that there is a cancellation in the order of the denominator of \eqref{E:jointcharacfcncalX}. That is, the leading order of each of $\overline{w}_{1,N+1}/(\Lambda_1\Lambda_2)$ and $\overline{q}_{N}/(\Lambda_1\Lambda_2) $ will be some order  $1$ trigonometric factor times $it/N$; in fact there holds $(1-a)\overline{q}_{N}/\overline{w}_{1,N+1} \sim 1$, as $N\to\infty$.  Define 
\begin{equation}\label{E:DELTA}
\Delta_N:=\overline{w}_{1,N+1}-(1-a)u_Nr_N^{2} z_N^{2} \overline{q}_{N}. 
\end{equation}
By direct calculation we will establish that $\Delta_N/(\Lambda_1\Lambda_2)=O(N^{-2})$, and we find the exact coefficient of the order $N^{-2}$ term. 
\par 
For the asymptotics of \eqref{E:dq1dq2reduced} 
we may still treat  $\sin\theta_1 =\theta_1+O(N^{-3})$ by \eqref{E:asymptheta}, but must render precisely the $O(N^{-2})$ term in $\cos\theta_1=1+O(N^{-2})$ of \eqref{E:asympcostheta}. In an appendix to \cite{Mor2018}, we display the many terms of the book--keeping method for this problem. For the the present, we simply exhibit the asymptotics of \eqref{E:DELTA} obtained by  machine computation with $\sin\theta_j$ substituted by the corresponding order $1/N$ term of \eqref{E:asymptheta}:
\begin{equation}\label{E:qNminusfasymp1update}
\begin{array}{l}
\frac{\Delta_N}{\Lambda_1\Lambda_2}=\frac{1}{(1-a)(1-b)}\frac{t^2}{N^{2}} \{ -a b \sigma_1\sigma_2\mathbf{c}_1\mathbf{c}_2-a\sigma_1(a-b)^2\mathbf{c}_1\mathbf{s}_2+a^2\sigma_1^2\mathbf{s}_1\mathbf{s}_2 \} +O\left(\frac{1}{N^3}\right).
\end{array}
\end{equation} 
\par
Finally we compute the limit of the ratio  \eqref{E:jointcharacfcncalX} by the asymptotic relations \eqref{E:thetaab}, and by \eqref{E:w0Nformulafinal1} and \eqref{E:qNminusfasymp1update}. Thus, because by \eqref{E:Hhomogeneous} and Proposition \ref{P:Pi} we have that $P(\mathbf{H}\ge N+1)\sim C_{a,b}N^{-1}$ for $C_{a,b}=ab/[(1-\eta)a+\eta b-ab]$, we find $E\{e^{it(1+1/N){\cal X}_{N+1}}\}$ is asymptotic to
 $$\begin{array}{l}
C_{a,b}N^{-1} \left \{[a\sigma_1\mathbf{c}_1\mathbf{s}_2+b\sigma_2\mathbf{s}_1\mathbf{c}_2+(b-a)^2 \mathbf{s}_1\mathbf{s}_2]\frac{it}{N}+O\left(\frac{1}{N^2}\right)\right \}\\  / \left\{ \frac{1}{(1-a)(1-b)}[-a b \sigma_1\sigma_2\mathbf{c}_1\mathbf{c}_2-a\sigma_1(a-b)^2\mathbf{c}_1\mathbf{s}_2+a^2\sigma_1^2\mathbf{s}_1\mathbf{s}_2 ]\frac{t^2}{N^2} +O\left(\frac{1}{N^3}\right)\right \}.
\end{array}$$
As in the proof of Theorem \ref{T:T2} we have $\mathbf{c}_j\sim \cosh(\kappa_j t)$, and $\mathbf{s}_j\sim i \sinh(\kappa_j t)$, $j=1,2$. Therefore, with  $\tilde{C}_{a,b}:=(1-a)(1-b)C_{a,b}$, we obtain that $E\{e^{it(1+1/N){\cal X}_{N+1}}\}$ has the following limit as $N\to \infty$, where we refer to \eqref{E:T2} and statement of Theorem \ref{T:T3} for the definitions of $\hat{\varphi}(t)$ and $\hat{\psi}(t)$: 
$$\lim_{N\to\infty}E\{e^{it(1+1/N){\cal X}_N}\}=\frac{\tilde{C}_{a,b}}{t}\times \frac{(b\kappa_1\sigma_2+a\kappa_2\sigma_1)t}{\hat{\varphi}(t)}\times \frac{\hat{\psi}(t)}{ab\sigma_1\sigma_2}.$$
We have  $\tilde{C}_{a,b}=ab\sigma_1\sigma_2/(a\sigma_1\kappa_2+b\sigma_2\kappa_1)$, so 
the proof is complete.
\qed 
\end{proof}
\begin{corollary}\label{C:lastvisitlaw}
Assume $a=b$. Define \\
$\begin{array}{l}
Z_{1}=\frac{1}{N}\left({\cal R}_N-\frac{1}{(1-a)}{\cal V}_N+a{\cal M}_N\right); 
Z_{2}=\frac{1}{N}\left({\cal L}_N-\frac{1}{(1-a)}{\cal R}_N+\frac{a}{1-a}{\cal M}_N\right) - Z_{1}.
\end{array}$
Then,
$\lim_{N\to\infty}E\{e^{i(sZ_{1}+tZ_{2})}\}= \frac{\tanh(\sqrt{(1-a)s^2+at^2})}{ \sqrt{(1-a)s^2+at^2}}.$
\end{corollary}
\begin{proof} One simplifies the lines of proof of Theorem \ref{T:T3}. We leave details in an appendix to \cite{Mor2018}. 
\qed 
\end{proof}
\begin{figure}[!]
  \centering
  \setlength{\unitlength}{\textwidth} 
    \begin{picture}(1,0.5)
       \put(-0.12,0){\includegraphics[width=14cm,height=6 cm]{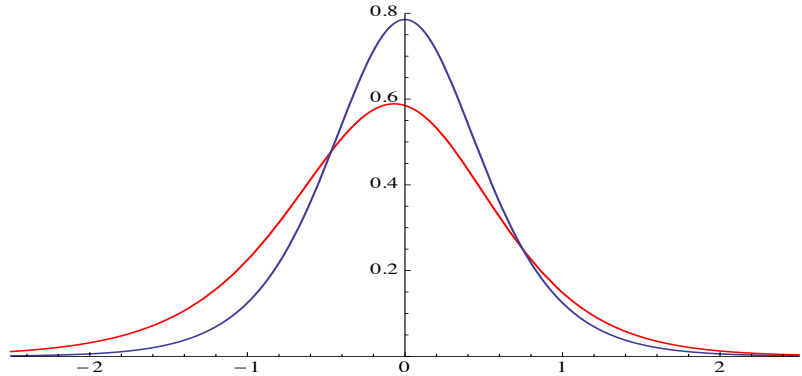}}
    \end{picture}
    \caption{The density $\varphi(x)$ whose transform $\hat{\varphi}(t)=\int_{-\infty}^{\infty}e^{itx}\varphi(x)\ dx$ is given  by \eqref{E:phihat} for $a=\frac{1}{4}$, and the density  $\frac{\pi}{4} \mbox{sech}^2(\pi x/2)$, that is instead determined by $a=\frac{1}{2}$ and corresponds to simple random walk. Numerically, the mean of $\varphi$ is  $\int_{\infty}^{\infty} x\varphi(x)\ dx=-\frac{1}{4}$, and $\mbox{arg} \max\limits_{x} \varphi(x)=-0.131619.$}
\label{F:density}
\end{figure}
\begin{acknowledgement} The author wishes to thank the referee who gave extensive suggestions that led to many improvements in the presentation. The companion document \cite{Mor2018} would not have come into the public domain without the referee's helpful (and exuberant!) insight.
\end{acknowledgement}
\bibliography{mybibfile}

\end{document}